\documentclass[10pt,a4paper]{article}
\usepackage[utf8]{inputenc}
\usepackage[english]{babel}
\usepackage{geometry}
\usepackage{amsmath}
\usepackage{amsfonts}
\usepackage{amssymb}
\usepackage{amsthm}
\usepackage{tikz-cd}
\usepackage{mathrsfs}
\usepackage{mathtools}
\usepackage{graphicx}
\usepackage{todonotes}
\usepackage{stmaryrd}
\usepackage{hyperref}
\usepackage{thmtools}
\usepackage[normalem]{ulem}

\newcommand{\A}{\mathscr{A}}
\newcommand{\Abb}{\mathbb{A}}
\newcommand{\B}{\mathbb{B}}
\newcommand{\C}{\mathbb{C}}
\newcommand{\D}{\mathbb{D}}
\newcommand{\R}{\mathbb{R}}
\newcommand{\Z}{\mathbb{Z}}

\newcommand{\K}{\mathbb{K}}

\newcommand{\im}{\mathop{\textup{im}}}
\newcommand{\colim}{\mathop{\textup{colim}}}
\newcommand{\dom}{\mathop{\textup{dom}}}

\newcommand{\F}{\mathscr{F}}

\let\L\temp

\renewcommand{\P}{\mathbb{P}}
\newcommand{\Q}{\mathbb{Q}}
\newcommand{\forces}{\Vdash}

\renewcommand{\r}{\rightarrow}

\newcommand\inv[1]{#1\raisebox{1.15ex}{$\scriptscriptstyle-\!1$}}

\newcommand{\stone}{\mathbf{Stone}}
\newcommand{\sol}{\mathbf{Sol}}
\newcommand{\pyk}{\mathbf{Pyk}}
\newcommand{\cond}{\mathbf{Cond}}

\newcommand{\VD}{\mathbf{VD}}
\newcommand{\I}{\mathbb{I}}
\newcommand{\set}{\mathbf{Set}}
\newcommand{\bprod}{\prod^{bdd}}
\renewcommand{\a}{\mathbf{a}}
\newcommand{\qo}{\text{quasi-open}}
\newcommand{\Abvd}{\mathbf{Ab}_{\VD}}
\newcommand{\dashhom}{\dashuline{Hom}}

\newcommand{\blhs}{Bergfalk--Lambie-Hanson--\v{S}aroch}
\newcommand{\Lev}{\mathbf{VD}}

\declaretheorem[numberwithin=section]{theorem}
\newtheorem{fact}[theorem]{Fact}
\newtheorem*{quotthm}{Theorem}

\newtheorem{ex}[theorem]{Example}
\newtheorem{cor}[theorem]{Corollary}
\newtheorem{subclaim}{Claim}[theorem]
\newtheorem{lemma}[theorem]{Lemma}
\newtheorem{prop}[theorem]{Proposition}
\newtheorem{defn}[theorem]{Definition}
\newtheorem{quest}[theorem]{Question}
\declaretheorem[style=definition,sibling=theorem]{remark}

\def\theqedlabel{\thetheorem\ifnum\arabic{subclaim}>0 .\arabic{subclaim}\fi}

\def\subclaimsdone{\setcounter{subclaim}{0}}

\title{\textbf{Condensed Sets and the Solovay Model}}
\author{\textsc{Nathaniel Bannister, Dianthe Basak}}

\begin{document}
\maketitle

\begin{abstract}
We exhibit a geometric morphism from the Grothendieck topos representing the Solovay model to the $\kappa$-pyknotic sets of Barwick--Haine and Clausen--Scholze. We then use the properties of this morphism and automatic continuity in the Solovay model to prove Clausen--Scholze's resolution of the Whitehead problem for discrete condensed abelian groups. We also exhibit an analogous internal $Ext$ computation between locally compact abelian groups in the Solovay model. 
\end{abstract}

\section*{Foreword}

The arguments of this paper cross back and forth through the looking glass between a structural (i.e. topos theoretic) and material (i.e. ZFC based) account of the theory of sets.

There is a vast body of work from innumerable sources that has gone into bridging this gap, through the development of topos theoretic accounts of forcing and symmetric extensions, and through the extensive network of powerful bi-interpretability results between the formalisms. This phenomenon of bi-interpretability is not the subject of the present paper, but we refer the interested reader to Blass--Scedrov \cite{blass-scedrov, rep-topoi}, Fourman \cite{fourman}, Hayashi \cite{hayashi}, Shulman \cite{shulman}, or, for a broader intuitionistic context, Awodey--Butz--Simpson--Streicher \cite{awodey-butz-simpson-streicher}. We simply note that in the Boolean setting, this bi-interpretability phenomenon rests on two slogans:
\begin{itemize}
\item A model of set theory gives rise to a Boolean complete (well-pointed, ...) topos by considering its maps.
\item A Boolean complete topos gives rise to a symmetric Boolean-valued model of set theory by considering its smallest exponential subvariety (its version of the cumulative hierarchy).
\end{itemize}

With these slogans in mind, we instead use the tools of the bi-interpretability phenomenon to focus on one specific topos, that of the $\kappa$-pyknotic (also called the $\kappa$-condensed) sets for an inaccessible cardinal $\kappa$, and study it by set theoretic means. In so doing we are guided by new definitions obtained by translating topos theoretic constructions.

We assume some familiarity with topos theoretic definitions such as the sheaf axiom, adjoint functors, Kan extensions, and geometric morphisms, though we attempt to define any types of geometric morphisms in use. All topos theoretic concepts left undefined may be found in Maclane--Moerdijk \cite{maclane-moerdijk}. However, we assume a heavier familiarity with forcing terminology, and with regularity properties of sets of reals, particularly in the Solovay model (an account may be found in Jech \cite{jech} or in Solovay's original paper \cite{solovay}), a bias that reflects the authors' backgrounds.

To practitioners of either field, this paper will at different points appear to reiterate some well-known tricks or techniques in the interest of exposition. Nevertheless, the hope is that this paper will help bring the two communities closer together, by highlighting some transfer results between the subjects.

\section{Introduction}

Clausen and Scholze \cite{condensedpdf} have proposed a formulation of rigid analytic geometry, and more broadly topological algebra, via the formalism of condensed sets. Independently, a small-generated version of this category, known as the $\kappa$-pyknotic sets, was proposed by Barwick and Haine \cite{pyk}. The condensed sets (resp. the pyknotic sets) are the category of sheaves on the site of all (resp. of $(< \kappa)$-sized) totally disconnected compact Hausdorff spaces, or equivalently on the full subcategory of extremally disconnected compact Hausdorff spaces. These sheaf categories permit well-behaved embeddings from the category of compactly generated (resp. $\kappa$-compactly generated) Hausdorff spaces, and this allows the categories to be seen as well-behaved completions of various categories of spaces.

Moreover, the close relationship with the category of compact Hausdorff spaces allows for the internal homological algebra to reflect continuous behaviour. In particular, computations of the internal extension group (i.e. $\underline{Ext}$) between locally compact abelian groups reduce to their continuous analogues. We highlight two such computations:

\begin{quotthm}[{Clausen--Scholze, \cite[Sessions 6 and 8]{clausen_lecture}}]
Let $\underline{A}$ be the condensed abelian group corresponding to a topological abelian group $A$, namely the sheaf
\[
\underline{A}(S) = \{\text{continuous }f \colon S \to A\}.
\]
\begin{itemize}
\item If $A$ is discrete and $\underline{Ext}(\underline{A}, \underline{\Z}) = 0$ in the category of condensed abelian groups, then $A$ is free.
\item $\underline{Ext}(\underline{\R}, \underline{\Z}) = 0$.
\end{itemize}
\end{quotthm}

\noindent These computations express the extent to which the Whitehead conjecture holds in the category of abelian group objects in the condensed sets, the first showing that it is true for discrete groups, and the second that it fails for $\R$.

The first result was given a forcing theoretic proof in \cite{WPCM} by \blhs, using the interpretation of condensed sets as ``generically invariant" information, or more accurately as a system of names over every forcing extension. This insight is based on the classical duality between continuous maps and names, namely that continuous maps from the Stone space of a complete Boolean algebra $\B$ into a compact Hausdorff space $X$ correspond to names for points of the \textit{compact reinterpretation} of $X$ in the forcing extension by $\B$:
\[\begin{tikzcd}
	{{\Big\{}\text{continuous }f \colon S(\B) \to X {\Big\}}} && {{\Big\{} \text{names }\tau \in V^\B, 1 \forces_\B \tau \in \overline{X}{\Big\}}},
	\arrow[from=1-1, to=1-3]
    \arrow[from=1-3, to=1-1]
\end{tikzcd}\]
where $\overline{X}$ denotes the set of $V^\B$ maximal filters of closed $V$-subsets of $X$. Taking this analogy further, \blhs\ construct a language of propositions that translate between the internal language of the condensed sets and forcing with a specific Boolean algebra $\B$. In an unpublished note, Reid Barton \cite{Barton_note} has shown that this translation can be implemented by a geometric morphism into the condensed sets from the category of sheaves on $\B$ with the double negation topology. We may equivalently define this as the category of $\B$-names with names for maps up to forced equality -- for a proof see Blass--Scedrov \cite[page 44]{blass-scedrov}, who attribute it to an unpublished result of Higgs.

In particular, condensed sets (resp. $\kappa$-pyknotic sets) carry simultaneous information about all forcing extensions (resp. all forcing extensions by forcings smaller than $\kappa$). A natural question, therefore, is whether a \textit{single} model of set theory can contain this information, and be used as the ``generous arena" (to use Maddy's \cite{Maddy} term) in which to perform the proofs of condensed mathematics.

This paper concerns itself with producing such a model, insofar as this is possible. There are immediate obstructions, as the internal languages of both $\cond$ and $\pyk$ are intuitionistic. Moreover, the double negation subtopos of $\pyk$ is $\mathbf{Set}$ and its inclusion as the subcategory of discrete spaces (essentially because the continuous map $X_{disc} \to X$ into a compactum from its discretisation is dense), and thus the double negation topos retains essentially no topological information at all. One approach, then, might be to produce instead a Boolean cover (for instance, following Barr \cite{Barr}) of the topos. While this is a sound approach, the Barr surjection is not often well-behaved from an internal logic point of view.

We instead follow a different approach, defining a site $\sol$ and exhibiting a much better behaved intermediate cover $Sh(\sol) \to \pyk$. Most strikingly, the double negation topos of $Sh(\sol)$ coincides with a well-understood model of set theory! Specifically, it is the topos representing the Solovay model $V(\R)$ computed in the Levy collapse extension that makes $\kappa$ the first uncountable cardinal (by ``representing" we mean in the technical sense defined by Blass--Scedrov \cite{rep-topoi}, in that the corresponding cumulative hierarchy defines $V(\R)$).

We thus produce, as the main contribution of this paper, a geometric morphism interpreting back and forth between $\pyk$ and the Solovay model $V(\R)$. By interpreting the projective solid abelian groups \cite{condensedpdf} via this functor, we show that a peculiar phenomenon emerges: the continuous extension group computations in $Sh(\sol)$ corresponding to those described above in $\pyk$ become consequences of the often noted \cite{STNG, karagila, ACGH} \textit{automatic continuity} phenomenon in the Solovay model. In the case of the result on discrete abelian groups, we are even able to derive the corresponding result by reflecting back to $\pyk$!

The covering map to $\pyk$ is the subject of section \ref{prelims_section}. We compute its double negation subtopos and identify it with the representing topos of the Solovay model in section \ref{bridge_section}. Next, we exhibit that discrete Whitehead groups in $\pyk$ are free using regularity properties of sets of reals in the Solovay model, in section \ref{whitehead_section}. Lastly, in section \ref{ext_section}, we compute the appropriate version of $\underline{Ext}(\underline{\R}, \underline{\Z})$ in the Solovay model, and show that it too vanishes.

\pagebreak
\section{Preliminaries} \label{prelims_section}

\begin{defn}[conventions]
Throughout the paper:
\begin{itemize}
\item We work in the theory 
\[
ZFC + \text{ there exists an inaccessible cardinal,}
\]
which is assumed to be satisfied by the ground model $V$.
\item $\kappa$ is a fixed inaccessible cardinal.
\item $\K = \overline{Coll(\omega, < \kappa)}$, the Dedekind--Macneille completion of the Levy collapse forcing which collapses $\kappa$ to $\omega_1$. 
\item For a Boolean algebra $\B$, we denote by $\B_p$ the localisation at $p$:
\[
\{q \in \B \mid q \leq p\}.
\]
which is a Boolean algebra with top element $p$. $\pi_p(x) = x \wedge p$ will denote the canonical surjection to it.
%\item For a group $G$ acting on a Boolean algebra $\B$ and fixing $p \in \B$, we define the notation
%\[
%\B^G_p = \{q \in \B \mid q \leq p, H \text{ fixes }q\}.
%\]
%\item $\G = Aut(\K)$ with the topology where the groups
%\[
%\{Fix(\B) : \B \text{ a complete subalgebra of }\K, |\B| < \kappa\}
%\]
%are a neighbourhood basis at identity, where $Fix(\B)$ denotes the pointwise stabiliser.
%\item $\K^H_p = \{q \in \K : q \leq p, H \subseteq Fix(q)\}$
%\item $K_\alpha = Coll(\omega, < \alpha)$, $k_\alpha = Coll(\omega, \alpha)$
\item \textbf{Small} always refers to structures of size $< \kappa$.
\item The word \textbf{topos} without qualification refers to a \textbf{Grothendieck topos}, i.e. a category of sheaves on a set-sized site.
\end{itemize}
\end{defn}

\subsection{Stone Duality and Iliadis Absolutes}
\label{stoneiliadis}

The following objects will be central to our analysis:
\begin{defn}
A \textbf{Stone space} is a totally disconnected compact Hausdorff space (i.e. a compact Hausdorff space with a basis of clopen sets).
\end{defn}
\begin{defn}
An \textbf{extremally disconnected compact Hausdorff space} is a compact Hausdorff space in which the closure of every open set is open. It is in particular a Stone space.
\end{defn}
\noindent Under Stone duality, such spaces arise as the duals to Boolean algebras.
\begin{fact}[Stone \cite{stone}]
We recall the following facts about Stone duality:
\begin{itemize}
\item For any Stone space $X$, its algebra of clopen sets forms a Boolean algebra $\B(X)$. To any continuous map $f \colon X \to Y$ corresponds the Boolean homomorphism $\B(Y) \to \B(X)$ given by the preimage $\B(f)(p) = \inv{f}p$.
\item For any Boolean algebra $\B$, $S(\B)$ denotes its space of ultrafilters, with a basis for a topology given by
\[
\{U \in S(\B) \mid p \in U\}.
\]
for $p \in \B$. This space is Stone, and any homomorphism $h \colon \B \to \C$ gives rise to a continuous map $S(\C) \to S(\B)$ given again by the preimage $S(h)(U) = h^{-1}U$.
\item The above two processes are inverse antiequivalences between the category of Stone spaces and continuous maps, and the category of Boolean algebras and homomorphisms.
\item The complete Boolean algebras and complete homomorphisms correspond, under the above duality, to the category of extremally disconnected Stone spaces and open maps.
\end{itemize}
\end{fact}

The interested reader may consult \cite{devries_duality} for a pleasing dictionary between many more dual classes of maps that have been discovered since. What will be important to us is that every construction in either category has a dual in the other: products of spaces become tensor products of algebras, disjoint unions of spaces become direct products of algebras, and so on. Just as tensor products of infinite complete Boolean algebras are not complete, products of infinite extremally disconnected Stone spaces are never extremally disconnected (see the beginning of \cite{product-note} for a proof attributed to Walter Rudin). The construction of most interest to us is the Iliadis absolute, which corresponds in the case of algebras to the Dedekind--Macneille completion.

\begin{defn}
For any compact Hausdorff space $X$, its \textbf{Iliadis absolute} $\A(X)$ is the Stone dual of the Boolean algebra $RO(X)$ of regular open subsets of $X$. It comes equipped with a continuous map $\pi_X \colon \A(X) \to X$ given by
\[
\pi_X(F) = \text{the unique point of }\bigcap \left\{\overline{U} \mid U \in F\right\}.
\]
(the intersection of the closures of the elements of the ultrafilter $F$ on $RO(X)$).
\end{defn}

\begin{fact}[{\cite[Lemma II.3.3]{Kunen}}]
When $X = S(\B)$ is the Stone space of a Boolean algebra, $\A(X) = S(\overline{\B})$ where $\overline{(\cdot)}$ denotes the Dedekind--Macneille completion.
\end{fact}

\noindent This construction is \textbf{not} in general functorial, though it is on a large class of maps:

\begin{defn}
A continuous map is \textbf{\qo{}} iff the image of every non-empty open has non-empty interior iff the inverse image of a dense set is dense.
%A continuous map is \textbf{generic} iff the inverse image of a dense open set is dense open iff the image of every open set is somewhere dense.
\end{defn}

\noindent We summarise the relevant facts that identify the importance of the above concept:
\begin{prop}[\cite{ponomarev-shapiro}]
\label{maps_facts}
Let $X$ and $Y$ be compact Hausdorff spaces.
\begin{itemize}
\item A map $f \colon X \to Y$ is \qo{} iff it is \textbf{skeletal}, i.e. iff the inverse image of an open dense set is open dense (see Lemma \ref{skeletal-is-qo}).
\item A \qo{} map between extremally disconnected compact Hausdorff spaces is also open (see Proposition \ref{ext-disc-gen-open}).
\item $\pi_X \colon \A(X) \to X$ is surjective, irreducible, and \qo{} (see Proposition \ref{pigen}).
\item Any continuous map $f \colon X \to Y$ has a continuous lift $\tilde{f} \colon \A(X) \to \A(Y)$ that fits in the following diagram (such a lift is called an \textbf{absolute} of $f$).
\[\begin{tikzcd}
	{\A(X)} && {\A(Y)} \\
	\\
	X && Y
	\arrow["{\tilde{f}}", from=1-1, to=1-3]
	\arrow["{\pi_X}"', from=1-1, to=3-1]
	\arrow["{\pi_Y}", from=1-3, to=3-3]
	\arrow["f", from=3-1, to=3-3]
\end{tikzcd}\]
\item If $f \colon X \to Y$ is \qo{}, such a lift is both unique and open, and is denoted $\A(f)$. Hence for \qo{} maps, this construction is functorial (see Proposition \ref{gen-absolutes-open}).
\end{itemize}
\end{prop}

Proofs may be found in the appendix. The \qo{} maps are a class that is just wide enough to make extremally disconnected spaces amenable to topological analysis. They allow us to construct ``product-like" and ``pullback-like" spaces, since for $X$ and $Y$ extremally disconnected, the projection map $X \times Y \to X$ is open, hence the composite
\[
\A(X \times Y) \xrightarrow{\pi_{X \times Y}} X \times Y \to X
\]
is a composite of \qo{} maps between extremally disconnected spaces, hence open, despite $X \times Y$ not being (in general) extremally disconnected.

The above construction is not quite as surprising when seen dually: it is the (completion of the) tensor of Boolean algebras.
\begin{defn}
\label{tensor}
For Boolean algebras $\P$ and $\Q$, we define their \textbf{tensor product} to be
\[
\P \otimes \Q = \B(S(\P) \times S(\Q)).
\]
\end{defn}
Note that this has a more explicit expression as the tensor product of rings: it is the algebra generated by formal pairs $p \otimes q$ with $p \in \P$, $q \in \Q$ subject to the relations
\[
(p \otimes q) \wedge (p' \otimes q') = ((p \wedge p') \otimes (q \wedge q'))
\]
and
\[
(p \vee p')\otimes q = (p \otimes q) \vee (p' \otimes q); \ p \otimes (q \vee q') = (p \otimes q) \vee (p \otimes q').
\]
In the interest of familiarity, we note that Givant and Halmos \cite{halmos} refer to this construction as the \textbf{sum of algebras}, while in the literature it is sometimes also called the \textbf{product forcing}.

\subsection{The Sites}
We now turn to the definitions of the topoi under consideration.
\begin{defn}[$\stone_\alpha$]
For $\alpha$ a strong limit cardinal, we define the site $\stone_\alpha$:
\begin{itemize}
\item \textbf{Objects:} Extremally disconnected Stone spaces of size $< \alpha$.
\item \textbf{Morphisms:} Continuous maps with composition of maps.
\item \textbf{Covers:} \textbf{Finite} collections of maps $\{f_i \colon X_i \to X\}$ whose images jointly cover $X$.
\end{itemize}
\end{defn}
\begin{defn}[$\pyk$, $\cond$]
The category of \textbf{$\alpha$-pyknotic sets} is the category of ($\mathbf{Set}$-valued) sheaves on the site $\stone_\alpha$:
\[
\pyk_\alpha = Sh(\stone_\alpha).
\]
The category $\mathbf{Cond}$ of condensed sets is the colimit of the categories $\pyk_\alpha$ along the left Kan extension functors $\mathbf{Pyk}_\alpha \to \mathbf{Pyk}_{\beta}$ for $\alpha < \beta$.

We will always stick to the $\kappa$-pyknotic sets ($\pyk_\kappa$) for our fixed inaccessible cardinal $\kappa$, so we will drop the subscript, writing $\stone$ for $\stone_\kappa$ and $\pyk$ for $\pyk_\kappa$.
\end{defn}

\begin{remark}
The choice to work with $\pyk$ rather than $\mathbf{Cond}$ is mainly with a view to using the theory of Grothendieck topoi unfettered (noting that $\mathbf{Cond}$, despite being a category of sheaves, is not even an elementary topos as it does not have a subobject classifier). Nevertheless, when $\kappa$ is a large cardinal, theorems from $\cond$ are sometimes able to pass ``upward from" or ``downward to" the $\pyk$. This phenomenon is useful in our applications (for instance see section \ref{whiteheadargumentfinal}), and set theoretically interesting in its own right, but we do not investigate it further in the present paper (however, see section \ref{questions} for some discussion on the matter).
\end{remark}
Note that although the objects of the site above are \textit{complete} Boolean algebras, the maps do not reflect this completeness (under Stone duality they are dual to potentially non-complete homomorphisms between the Boolean algebras). In the interest of using the structure of the Boolean algebras more closely, we were led to consider the following site, whose maps are dual to complete Boolean algebra homomorphisms.

\begin{defn}[$\sol$, $\sol_{\neg\neg}$]
\label{soldef}
We define the site $\sol$ similarly, retaining only open maps:
\begin{itemize}
\item \textbf{Objects:} Extremally disconnected Stone spaces of size $< \kappa$.
\item \textbf{Morphisms:} \textbf{Open} continuous maps with composition of maps.
\item \textbf{Covers:} \textbf{Finite} collections of maps $\{f_i \colon X_i \to X\}$ whose images jointly cover $X$.
\end{itemize}
We will mainly study $Sh(\sol)$, the category of sheaves on the above site, and its double negation subtopos $Sh(\sol_{\neg\neg})$, constructed on the same category as above but with
\begin{itemize}
\item \textbf{Covers:} \textbf{Arbitrary} collections of maps $\{f_i \colon X_i \to X\}$ such that $\bigcup_i f_i[X_i]$ is \textbf{dense} in $X$.
\end{itemize}
This is also called the \textbf{dense coverage} on the above category.
\end{defn}

\noindent For readers who are more comfortable thinking in terms of complete Boolean algebras and complete Boolean homomorphisms, we observe that the covers in $Sh(\sol_{\neg\neg})$ admit the following set theoretic reformulation:
\begin{fact} \label{negnegcovers-fact}
    Suppose $\langle f_i\colon\B\to\C_i\mid i<\alpha\rangle$ are complete Boolean homomorphisms between algebras of size less than $\kappa$. The following are equivalent:
    \begin{itemize}
        \item $\langle S({f_i})\colon S(\C_i)\to S(\B)\mid i<\alpha\rangle$ defines a cover in $\sol_{\neg\neg}$;
        \item $\prod_if_i\colon\B\to\prod_i\C_i$ is a complete embedding;
        \item $\K$ forces that if $I$ is a $V$-generic filter over $\B$, then there is an $i<\alpha$ and a $V$-generic filter $J$ over $\C_i$ such that $I=f_{i}^{-1}J$. 
    \end{itemize}
\end{fact}

The category of extremally disconnected spaces and open maps is more difficult to work with than the category where maps are merely continuous, as the former does not possess an obvious enlargement in which pullbacks or even products might be computed (for instance Bezhanishvili--Kornell \cite{Bezhanishvili-Kornell} show that not even the category of all topological spaces and open maps can have a product object). With this in mind, even that this data \textit{is} a site is not obvious.

\begin{lemma}
Both classes of covers in definition \ref{soldef} satisfy the coverage axioms. Hence both sets of data do indeed give sites.
\end{lemma}
\begin{proof}
We refer the reader to \cite[Chapter III, Exercise 3]{maclane-moerdijk} for a list of the coverage axioms. The presence of isomorphisms is clear, as is compositionality for the first site. For the second site, it follows from noting that the forward image of a dense open set under an open map with dense image is dense open.

The third axiom cannot be satisfied in the form of the existence of pullbacks, as these do not exist in general for open maps. However, the pullback of an open map along another open map in the category of continuous maps between Stone spaces exists and is open. This lets us construct, for a cover $\{U_i\}_i \xrightarrow{\{f_i\}_i} X$ and open map $Y \xrightarrow{g} X$, the pullback diagram in the continuous category
\[\begin{tikzcd}
	&& {\{U_i\}_i} && X \\
	\\
	{\{\A(V_i)\}_i} && {\{V_i\}_i} && Y
	\arrow["{\{f_i\}_i}", from=1-3, to=1-5]
	\arrow["{\{\pi_{V_i}\}_i}", from=3-1, to=3-3]
	\arrow["{\{g'_i\}_i}", from=3-3, to=1-3]
	\arrow["{\{f'_i\}_i}", from=3-3, to=3-5]
	\arrow["g"', from=3-5, to=1-5]
\end{tikzcd}\]
where the family $\{\A(V_i)\}_i \xrightarrow{\{f'_i \circ \pi_{V_i}\}_i}_i Y$ is now in the category of extremally disconnected Stone spaces and open maps (since the composites are \qo{}, and the domain and codomain are extremally disconnected). We will show that in either case, this family is a cover of $Y$ in the relevant coverage.

In the finite case, $\bigcup_i f'_i[V_i] = Y$, hence $\bigcup_i f'_i \circ \pi_{V_i}[\A(V_i)] = Y$. To see that we still have a cover in the infinite dense case, note that
\[
\bigcup_i f'_i \circ \pi_{V_i}[\A(V_i)]
\]
is an open subset of $Y$. Further, if it is not dense, there is an open subset $W \subseteq Y$ which does not intersect it. But then, by the definition of pullbacks in the continuous category, $g[W]$ must not intersect the image of $\bigcup_i f_i[U_i]$. But this is impossible, since $g[W]$ is open, and $\bigcup_i f_i[U_i]$ is dense. 
\end{proof}

\begin{remark}
We note that the above is really a proof of the amalgamation property for complete Boolean algebras, but conducted in the dual category. The dual proof may be found in Monro \cite{Monro_AP}, who shows also a stronger form of the amalgamation property. Many amalgamation properties of Boolean algebras seem to have immediate counterparts in the theory of $Sh(\sol)$.
\end{remark}

Thus we may consider the categories of sheaves on $\sol$ and $\sol_{\neg\neg}$ respectively, which are more permissive than $\pyk$ (since the data is not required to have an action of \textit{all} continuous maps). Many set theoretically useful definitions are indeed invariant under the action of open maps, and give rise to examples of sheaves on the above site.

\begin{ex}
\label{exub}
For every compact Hausdorff space $X$, the following defines a sheaf on $\sol$ and on $\sol_{\neg\neg}$:
\[
X^o(S(\B)) = \{f \colon S(\B) \to X \mid f \text{ \qo{}}\}.
\]
A proof of the sheaf axiom may be found in the appendix. Similarly, we have sheaves on $\sol$ ``coming from $\pyk$" defined by
\[
\underline{X}(S(\B)) = \{f \colon S(\B) \to X \mid f \text{ continuous}\},
\]
\[
F_{\kappa\text{-uB}}(S(\B)) = \{X \subseteq S(\B) \mid X \text{ is $\kappa$-universally Baire}\}.
\]
Showing that these are sheaves on $\stone$ is routine. We will see that these are sheaves of $\sol$ in section \ref{bridge_section}, once we show that the functor which ``forgets the action of non-open maps" preserves the sheaf axiom.
\end{ex}

We will show that we are able to interpret $Sh(\sol_{\neg\neg})$ using the so-called Solovay model. This is the point at which topos theory makes its transition into material set theory, and is made explicit by the topos we denote $\VD$. Recall from the preliminaries that $\K = \overline{Coll(\omega, < \kappa)}$, i.e. the completion of the Levy collapse forcing that makes $\kappa = \omega_1$ in the extension.

\begin{defn}[Solovay model]
Denote by $V^\K$ a forcing extension of $V$ using $\K$. In $V^\K$, we may construct the model
\[
V(\R) = \bigcup_\alpha \L(V_\alpha, \R)
\]
which will be referred to as the \textbf{Solovay model}.
\end{defn}

 Note that this is \textbf{not} the model considered in \cite{solovay} by Solovay, which is instead $HOD(Ord^\omega)$ in $V^\K$. Over time, many other models (most prominently $HOD(\R)$ and $\L(\R)$), have come to be known as the/a Solovay model, each built as a definable hull within the Levy collapse extension. An axiomatisation of some of these models with an appropriate choice of $V$ was recently put forward by Zapletal \cite{IRSM1}. Our model is the largest such model, and is in some sense suggested by the symmetric extension used to produce any such model. The following is its topos theoretic analogue.

\begin{defn}[$\Lev$] \label{Lev_def}
    We define $\Lev$ to be the category where
    \begin{itemize}
        \item Objects are $\mathbb{K}$-names $\dot a$ such that 
        \begin{itemize}
            \item for some parameter $b$ and formula $\psi(x,y)$,
            \[
            \Vdash_\K\psi(x,\check{b})\text{ defines }\dot a
            \]
            \item $\Vdash_\mathbb{K}\dot a\in V(\R)$.
        \end{itemize}
        \item A morphism from $\dot a$ to $\dot b$ is a name $\dot f$ for a function from $\dot a$ to $\dot b$ such that for some formula $\psi(x,y)$ and parameter $c \in V$, $\Vdash\psi(x,\check{c})$ defines $\dot f$. 
        We identify two names for functions if they are forced to be equal. 
    \end{itemize}
\end{defn}

 Several remarks are in order about the category $\VD$. First, we should, as usual, be careful about defining definability, but since we always allow parameters from the ground model, the usual tricks with the reflection theorem apply; see \cite[Theorem 7.4]{Kunen}. Second, the notation $\VD$ hides the restriction we have made to $V(\R)$. Specifically, we are not considering \textit{all} $V$-definable subsets of $V^\K$, but instead those $V$-definable sets that are also \textit{hereditarily definable from reals} and parameters in $V$. It is also helpful to note the following fact, which was pointed out to us by Asaf Karagila, and simplifies certain aspects of the analysis.
\begin{lemma}
$V(\R)$ is definably closed in $V^\K$, i.e.
\[
V(\R) = \bigcup_\alpha HOD(V_\alpha, \R)
\]
In particular, any rank-bounded definable subclass of $V(\R)$ is a set of $V(\R)$.
\end{lemma}
\begin{proof}
    The inclusion $V(\R)\subseteq \bigcup_\alpha HOD(V_\alpha, \R)$ is a trivial induction on the constructibility hierarchy for $\L(V_\alpha,\R)$. 
    For the reverse inclusion, suppose $X\in HOD(V_\alpha,\R)$ for some limit ordinal $\alpha$. 
    We first argue for the special case of $X\subseteq Ord\times V_\alpha\times\R$. 
    Fix a formula $\varphi$ and parameter $a\in Ord\times V_\alpha\times\R$ such that 
    \[X=\{b\mid \varphi(b,a)\}.\]
    By homogeneity of $\K$,
    \[X=\{b\,\mid \,\Vdash^{V[a,b]}_\K\varphi(b,a)\}.\]
    By the reflection theorem, there is a $\gamma>\kappa$ such that 
    \[X=\{b\,\mid \,\Vdash^{V_\gamma[a,b]}_\K\varphi(b,a)\},\]
    which is in $\L(V_\gamma,\R)$.

    For the more general case, a definable surjection \[Ord\times V_\alpha\times\R\to \text{formulae}\times(Ord\times V_\alpha\times\R)^{<\omega}\] yields a definable surjection $Ord\times V_\alpha\times\R\to \operatorname{OD}(V_\alpha,\R)$. 
    Then for any $X\in\operatorname{HOD}(V_\alpha,\R)$, there is an $\operatorname{OD}(V_\alpha,\R)$ subset of $Ord\times V_\alpha\times\R$ coding a structure isomorphic to $(\operatorname{tc}(X),\in)$. 
    The Mostowski collapse lemma combined with the special case then implies $X\in V(\R)$.
\end{proof}

The main fact about this model that we will be using is the following, which can be obtained by minor variations of the argument in Solovay \cite{solovay}. An exposition of this argument may be found in Jech \cite[Theorem 26.14]{jech} with Solovay's original model, or derived from Ikegami \cite[Corollary 1.6]{ikegami}, whose definition of the Solovay model matches ours.

\begin{fact}[essentially Solovay \cite{solovay}] \label{measurable_fact}
Every set of reals in the Solovay model has the property of Baire, is Lebesgue measurable, and satisfies the perfect set dichotomy.
\end{fact}

This gives rise to many tameness phenomena associated with sets of reals, and in particular certain cases of automatic continuity. For instance, every homomorphism between Polish groups in the Solovay model is continuous (see \cite{Pettis}). By severely constraining the expressive power of the ambient model $V^\K$, the Solovay model allows to remain only the ``easily definable" maps, sets, and constructions. This is what brings it into close contact with categories of continuous maps, which are in some sense easier to define.

\subsection{Automorphisms, Names for Generics, and the Like} \label{auts_subsection}

Complete Boolean homomorphisms provide a remarkable tool for analyzing forcings. 
The following basic fact will be useful throughout. 

\begin{fact}[folklore] \label{homs_names_fact}
    If $\Abb$ and $\B$ are complete Boolean algebras, there is a bijective correspondence between 
    \begin{itemize}
        \item 
    $\Abb$ names for $V$-generic filters over $\B$ up to forced equality
    \item complete Boolean homomorphisms from $\B$ to $\Abb$
    \end{itemize}
    where an $\Abb$-name $\dot H$ for a generic filter over $\B$ is mapped to the homomorphism
    \[
    p \mapsto \llbracket p \in \dot H \rrbracket
    \]
    while a homomorphism $\varphi \colon \B \to \Abb$ is mapped to the name
    \[
    \{(\varphi(p), \check{p}) \mid p \in \B\},
    \]
    which is a name for the preimage of the canonical generic over $\Abb$.
\end{fact}
If $\varphi\colon\Abb\to\B$ is a complete Boolean homomorphism, we obtain a map from $\Abb$ names to $\B$ names by transfinite recursion on rank, letting $\varphi(\dot x)=\{(\varphi(p),\varphi(\dot y))\mid(p,\dot y)\in\dot x\}$, which is a name for $\dot x(\varphi^{-1}\dot G)$ for $\dot G$ the canonical $\B$-name for a $\B$-generic filter. 
Restricting to automorphisms of a fixed complete Boolean algebra $\B$, we obtain an action on the set of $\B$ names $\operatorname{Aut}(\B)\curvearrowright V^{\B}$. 
This action has the remarkable property due originally to Miller of preserving truth: if $p\in\P$, $\varphi(\vec a)$ is a formula in the forcing language, and $\sigma\in\operatorname{Aut}(\B)$, 
\[p\Vdash\varphi(\vec a)\text{ iff }\sigma(p)\Vdash\varphi(\sigma(\vec a));\]
see for instance, \cite[Lemma 14.37]{jech}.
\begin{defn}
    A complete Boolean algebra $\B$ is \emph{homogeneous} if for all $p,q\in\B$, there is $\sigma\in\operatorname{Aut}(\B)$ such that $p$ and $\sigma(q)$ have a common lower bound. 
\end{defn}
A standard fact is that if $\B$ is homogeneous, then for all $p\in\B$, and all formulae $\varphi$ with parameters from the ground model, $p\Vdash\varphi$ if and only if $1_\B\Vdash\varphi$.
The collapse algebra $\K$ satisfies an extremely strong form of homogeneity; see, for instance, \cite[Theorem 26.12]{jech} for a proof of the following proposition. 
\begin{prop}
    $\K$ is \textbf{ultrahomogeneous} in the sense that if $\Abb,\B$ are complete subalgebras of $\K$ with $|\Abb|,|\B|<\kappa$ and $\sigma\colon\Abb\to\B$ is an isomorphism, there is $\widetilde{\sigma}\in\operatorname{Aut}(\K)$ extending $\sigma$.
\end{prop}

\subsection{Geometric Morphism Theory Crash Course}

\begin{defn}[Kan extensions]
If $F$ is a contravariant $\set$-valued functor on $\mathbf{A}$ and $h \colon \mathbf{A} \to \mathbf{B}$ is any functor between small categories, we define the \textbf{left Kan extension} $Lan_h F \colon \mathbf{B}^{op} \to \set$ as
\[
(Lan_h F)(X) = \colim_{X \to h(Y)} F(Y)
\]
where the colimit is over the comma category $const_X \downarrow h$, and the \textbf{right Kan extension} $Ran_h \colon \mathbf{B}^{op} \to \set$ as
\[
(Ran_h F)(X) = \lim_{h(Y) \to X} F(Y)
\]
where the limit is over the comma category $h \downarrow const_X$.
\end{defn}

Any functor $h$ as above induces a pullback functor between presheaf categories
\[
Psh(\mathbf B) \xrightarrow{h^*} Psh(\mathbf A)
\]
obtained by precomposing with $h$. Since limits and colimits in presheaf categories are taken objectwise, this functor is both continuous and cocontinuous, and since $\mathbf{A}$ and $\mathbf{B}$ are small, by the adjoint functor theorem $h^*$ permits both a left and right adjoint, which turn out to be the left and right Kan extensions respectively. $(h^*, Ran_h)$ is thus a geometric morphism between the presheaf topoi where $h^*$ has a further left adjoint. This setup arises often in nature, and we will encounter it again in section \ref{geom_morph}. We also set out a catalogue of notions pertaining to geometric morphisms which will be of use.

\begin{defn}[catalogue of geometric morphisms]
A \textbf{geometric morphism} $f \colon \mathbf{S_1} \to \mathbf{S_2}$ between two topoi is a pair of adjoint functors
\[\begin{tikzcd}
	{\mathbf{S_1}} && {\mathbf{S_2}}
	\arrow[""{name=0, anchor=center, inner sep=0}, "{f_*}"', shift right=2, from=1-1, to=1-3]
	\arrow[""{name=1, anchor=center, inner sep=0}, "{f^*}"', shift right=2, from=1-3, to=1-1]
	\arrow["\dashv"{anchor=center, rotate=-90}, draw=none, from=1, to=0]
\end{tikzcd}\]
where $f^*$, the left adjoint, preserves finite limits (note that by the adjoint functor theorem $f_*$ and $f^*$ already preserve set-sized limits and colimits respectively).
\begin{itemize}
\item A \textbf{localic geometric morphism} is one where every object in $\mathbf{S_1}$ is a subquotient of an object in the image of $f^*$.
\item A \textbf{surjective geometric morphism} is one for which $f^*$ is faithful. In particular, what will be called a \textbf{cover} $\mathbf S_1 \to \mathbf S_2$ is really the inclusion of $\mathbf S_2$ in $\mathbf S_1$ as a subcategory.
%\item An \textbf{essential geometric morphism} is one for which $f^*$ possesses a further left adjoint (denoted $f_{!}$). Note that this is equivalent to $f^*$ preserving all set-sized limits.
\item A \textbf{terminally connected geometric morphism} is one for which $f^*$ possesses a further left adjoint $f_!$ which furthermore preserves the terminal object ($f_!(1) \simeq 1$).
\end{itemize}
\end{defn}

These notions will be of immediate use when defining the ``bridge" between $\pyk$ and $\VD$, as they quantify the ways in which a geometric morphism can preserve information -- and indeed \textit{theorems} -- between topoi (though the internal logic will not be discussed in detail in this paper). What will be most relevant to us is the fact that, since $f^*$ preserves finite limits and colimits, it preserves exact sequences when interpreted as a functor
\[
\mathbf{Ab}_{\mathbf{S_2}} \to \mathbf{Ab}_{\mathbf{S_1}}
\]
between the categories of abelian group objects (in other words, $f^*$ is an \textbf{exact functor}).

\begin{defn}
If $\mathbf{S}$ is a topos, then its category of abelian group objects $\mathbf{Ab}_\mathbf{S}$ is the category of objects $c$ of $\mathbf{S}$ with $\mathbf{S}$ morphisms $(\cdot + \cdot) \colon c \times c \to c$, $(-) \colon c \to c$ and $0 \colon 1 \to c$ satisfying the abelian group axioms stated as equalities between morphisms.
\end{defn}

Since all our toposes are Grothendieck, $\mathbf{Ab}_\mathbf{S}$ will also always be a \textbf{Grothendieck abelian category}. An excellent exposition complete with a variety of tools for handling Grothendieck abelian categories (specifically those made of sheaves) is available in Kashiwara--Schapira \cite{kashiwara-schapira}.

% \begin{defn}
% An \textbf{exponential variety} in a topos is a full subcategory closed under products, subobjects, and power objects (hence also coproducts, colimits, limits, ...). The minimal exponential variety in a topos is referred to as its \textbf{well-founded part}. A \textbf{well-founded topos} is one which is its own well-founded part.
% \end{defn}
\pagebreak
\section{Bridge Building} \label{bridge_section}

We build a bridge gradually, starting at $\pyk$ constructing the left adjoints of various geometric morphisms until we reach $\VD$. The argument splits up quite cleanly into the first part, which is mainly topology (with some arguments about Boolean algebras) and the second, which is mainly set theory. We then turn to computing the images of some familiar objects in $\pyk$ under this transfer functor.

\subsection{$Sh(\sol)$ covers $\pyk$ Very Well}
\label{geom_morph}
The goal of this section is to prove

\begin{theorem}
\label{coverthm}
There is a terminally connected localic surjection from $Sh(\sol)$ to $\pyk$.
\end{theorem}

\noindent We have a clear inclusion, which we will denote $i$, of $\sol$ into $\stone$.
\begin{lemma} \label{i_pres_refl_lem}
The inclusion functor $i \colon \sol \to \stone$ both preserves and reflects covers.
\end{lemma}
\begin{proof}
Let $\{f_i \colon V_i \to U\}_i$ be a cover in $\sol$, that is, $f_i$ is a finite collection of open maps whose images cover $U$. It is in particular a cover of $U$ in $\stone$.

Let $C = \{f_i \colon V_i \to U\}_i$ be a cover in $\stone$, i.e. a finite collection of continuous maps whose images cover $U$. The map $f \colon \coprod_i V_i \to U$ that is the sum of the $f_i$ is a continuous map onto $U$, hence has a section $s \colon U \to \coprod_i V_i$. Define
\[
W_i = s(U) \cap V_i
\]
which is closed in $V_i$.

\begin{subclaim}
The $W_i$ are extremally disconnected.
\end{subclaim}
\begin{proof}
$s \colon U \to s(U)$ is a homeomorphism, hence is open \textbf{onto its image}. $V_i$ are clopen subsets of $\coprod_i V_i$, hence $V_i \cap s(U)$ are clopen subsets of $s(U)$. These are homeomorphic to clopen subsets of $U$, which is extremally disconnected. Therefore the $V_i \cap s(U) = W_i$ are extremally disconnected.
\end{proof}

\begin{subclaim}
$f_i\upharpoonright W_i \colon W_i \to U$ is open for each $i$.
\end{subclaim}
\begin{proof}
Let $P$ be open in $W_i$, then $P$ is open in $\coprod_i W_i$, and $f_i(P) = s^{-1}P$. The map $s \colon U \to \coprod_i W_i$ is continuous, hence $f_i(P)$ is open.
\end{proof}
\subclaimsdone
The images of $f_i\upharpoonright W_i$ are clearly $s^{-1}V_i$, hence are covering. Further, each $f_i\upharpoonright W_i$ is in the sieve generated by $C$. Hence $\{f_i\upharpoonright W_i \colon W_i \to U\}_i$ is in the image of the inclusion functor from $\sol$, and refines $C$. Thus $C$ is reflected.
\end{proof}

This functor gives rise to the pullback functor $i^*$ obtained by precomposing with $i$ (i.e. by ``forgetting" the action by non-open maps). Recall that we have, without any assumptions at all on $i^*$, an adjoint triple of functors
\[\begin{tikzcd}
	{Psh(\stone)} && {Psh(\sol)}
	\arrow["{i^*}"{description}, from=1-1, to=1-3]
	\arrow["{Ran_i}"{description}, shift left=3, from=1-3, to=1-1]
	\arrow["{Lan_i}"{description}, shift right=3, from=1-3, to=1-1]
\end{tikzcd}\]
Since $i$ reflects covers, $i$ gives rise to a comorphism of sites \cite[Theorem VII.10.5]{maclane-moerdijk}, hence $Ran_i$ descends to sheaves (with left adjoint the composite of $i^*$ followed by sheafification). We will now show, however, that $i$ is well enough behaved for both $i^*$ and $Lan_i$ to descend to the sheaf categories.

\begin{lemma} \label{i*_exists_lem}
If $F$ is a pyknotic set, $i^*F$ is a sheaf on $\sol$.
\end{lemma}
\begin{proof}
    Let $\{f_i\colon V_i\to U\}_i$ be a cover in $\sol$. 
    It suffices to show that if $\{x_i\in F(V_i)\}_i$ define a matching family in $\sol$, then $x_i$ also defines a matching family in $\pyk$. 
    The crucial point is that in the category of spaces and continuous maps, the pullback of an open map is again an open map, so 
    \[\pi_1,\pi_2\colon \left(\coprod_{i\in I}V_i\right)\times_U\left(\coprod_{i\in I}V_i\right)\rightrightarrows \coprod_{i\in I}V_i\]
    are each open maps. 
    Letting $E$ denote the pullback above, we therefore obtain obtain the following commutative diagram of Stone spaces with all maps appearing \qo{}
    %and where 
    %\[X=\left(\coprod_{i\in I}V_i\right)\times_U\left(\coprod_{i\in I}V_i\right):\]

\begin{center}
   % https://tikzcd.yichuanshen.de/#N4Igdg9gJgpgziAXAbVABwnAlgFyxMJZARgBpiBdUkANwEMAbAVxiRAA0QBfU9TXfIRQBmUsKq1GLNgFVuvEBmx4CRUZWr1mrRCAA6egMYQ0AJ2gB9YFgNYwAAgCSXAGoWs8vssFEy4zVI6+kYm5lBWNnp2Tq7unor8KkLIAAykKRJa0roGALZ0OAAWcIamwACCXAAU7ACU3BIwUADm8ESgAGbmuUiiIDgQSMQ8nd1IAEzUA0MjIF0QPYhp-YOIkyAMdABGMAwACok+uqZYzYU4IAHabAZoWBbE8fOLy9OIfVlBt-fjT2OIABYpqtlp8bno7hZOFwKFwgA
\begin{tikzcd}
\mathscr{A}(E) \arrow[rd, "\pi_E"] &                                            &  &                                \\
                                   & E \arrow[rr, "\pi_1"'] \arrow[dd, "\pi_2"] &  & \coprod_{i\in I}V_i \arrow[dd] \\
                                   &                                            &  &                                \\
                                   & \coprod_{i\in I}V_i \arrow[rr]             &  & U                             
\end{tikzcd}
\end{center}
Since $x_i$ form a matching family in $\sol$ and $F\left(\coprod_{i\in I}V_i\right)\cong\prod_iF(V_i)$, we see that \[F(\pi_1\circ\pi_E)(\langle x_i\mid i\in I\rangle)=F(\pi_2\circ\pi_E)(\langle x_i\mid i\in I\rangle)\]
since $\pi_E$ is a cover in the category of Stone spaces, $F(\pi_E)$ is injective, so 
\[F(\pi_1)(\langle x_i\mid i\in I\rangle)=F(\pi_2)(\langle x_i\mid i\in I\rangle).\]
Since $E$ is the fiber product of Stone spaces, we conclude that the $x_i$ form a matching family for the site $\stone$. In particular, $F$ has a unique gluing for this family. 
\end{proof}

\begin{lemma}
For any $F \in Sh(\sol)$, $Lan_i F$ is already a pyknotic set.
\end{lemma}
\begin{proof}
Note that in $\stone$, since any surjection splits, verifying the sheaf condition amounts to checking that disjoint unions are taken to products:
\begin{itemize}
	\item $Lan_i F(\{\}) = 1$ is immediate (the empty map is open).

	\item $Lan_i F(S(\B_1) \coprod S(\B_2)) = \colim_{S(\B_1 \times \B_2) \xrightarrow{cont.} S(\Abb)} F(S(\Abb))$.

	Notice that we can factor
	\[\begin{tikzcd}
		{\Abb \times \Abb} && {\mathbb B_1 \times \mathbb B_2} \\
		\Abb
		\arrow["{(f\pi_1, g\pi_2)}", from=1-1, to=1-3]
		\arrow["{ \Delta}", from=2-1, to=1-1]
		\arrow["{(f, g)}"', from=2-1, to=1-3]
	\end{tikzcd}\]
	where $\Delta$, the diagonal map, is complete, hence by replacing the diagram with a cofinal diagram, we can reason that this is
	\[
	\simeq \colim_{S(\B_1) \xrightarrow{cont.} S(\Abb)} \colim_{S(\B_2) \xrightarrow{cont.} S(\Abb')} F(S(\Abb \times \Abb'))
	\]
	\[
	\simeq Lan_i F(S(\B_1)) \times Lan_i F(S(\B_2))
	\]
    Abstractly, we are using that the inclusion of maps of the form
    \[
    \Abb \times \Abb' \xrightarrow{f\pi_1, g\pi_2} \B_1\times \B_2
    \]
    into maps $\Abb \xrightarrow{f, g} \B_1 \times \B_2$ is a final functor.
\end{itemize}
\end{proof}
\noindent In summary we have a triple
\[\begin{tikzcd}
	\pyk && {Sh(\sol).}
	\arrow["{i^*}"{description}, from=1-1, to=1-3]
	\arrow["{Ran_i}"{description}, shift left=3, from=1-3, to=1-1]
	\arrow["{Lan_i}"{description}, shift right=3, from=1-3, to=1-1]
\end{tikzcd}\]
Since $i^*$ has a left and right adjoint, it is continuous and cocontinuous (respectively), hence in particular $(i^*, Ran_i)$ are the left and right part of a geometric morphism $Sh(\sol) \to \pyk$. We now show that this morphism is quite special, namely that:
\begin{itemize}
\item $i^*$ is faithful ($(i^*, Ran_i)$ is a surjection).
\item Every object of $Sh(\sol)$ is a subquotient of one in the image of $i^*$ ($(i^*, Ran_i)$ is localic).
\item $Lan_i$ preserves the terminal object ($(i^*, Ran_i)$ is terminally connected).
\end{itemize}

\begin{lemma}
$i^*$ is faithful.
\end{lemma}
\begin{proof}
$\sol$ and $\stone$ have the same objects, hence $i^*$ has no effect on (the components of) natural transformations. Thus if $i^*\eta_1 = i^*\eta_2$, we must already have had $\eta_1 = \eta_2$.
\end{proof}

In order to establish localicness, we note from the appendix Proposition \ref{subcanonical_prop}, that the representables
\[
X^o(S) = \{f \colon S \to X \mid f\text{ open continuous}\}
\]
are sheaves on $\sol$, and that $X^{o} \hookrightarrow i^*\underline{X}$.

\begin{lemma}
$(i^*, Ran_i)$ is localic.
\end{lemma}
\begin{proof}
For each extremally disconnected compact Hausdorff space $X$, $X^{o}$ is a subobject of $i^*\underline{X}$. Note that every object of $Sh(\sol)$ is a quotient of some sufficiently large coproduct of representables, hence a subquotient of a large coproduct of some family $\{i^*\underline{X_\alpha}\}_\alpha$ where each $X_\alpha$ is extremally disconnected compact Hausdorff. However $i^*$ preserves coproducts, hence this coproduct is in the image of $i^*$.
\end{proof}

\begin{lemma}
$Lan_i(1) \simeq 1$
\end{lemma}
\begin{proof}
We compute directly:
\[
Lan_i(1)(S(\B)) = \{[f \colon S(\B) \to X] \mid f \text{ continuous}, [gf] = [f] \text{ if }g \colon X \to Y\text{ is open}\}
\]
However since $X \to 1$ is open, $[S(\B) \xrightarrow{f} X] = [S(\B) \xrightarrow{f} X \to  1]$, hence $Lan_i(1) \simeq 1$.
\end{proof}

This completes the proof of Theorem \ref{coverthm}. It is reasonable to ask whether more can be shown about this morphism. We note for completeness that $Lan_i$ does \textbf{not} preserve other limits, such as equalisers.\footnote{An earlier draft of this paper asserted that $Lan_i$ preserved products, but the proof contained an error. See Lemma \ref{product_map_epi} for a more modest claim that can be shown.}

\begin{prop}
Let $X$ be extremally disconnected compact Hausdorff that lacks isolated points, and let $Y = \A(X \times X)$ be the absolute of its product with itself, with its two open projection maps
\[
\pi_1, \pi_2 \colon \A(X \times X) \to X \times X \rightrightarrows X
\]
The equaliser of $\pi_1, \pi_2 \colon Y^o \rightrightarrows X^o$ in $Sh(\sol)$ is the initial object, whereas the equaliser of $\pi_1, \pi_2 \colon \underline{Y} \rightrightarrows \underline{X}$ in $\pyk$ is the preimage of the diagonal in $\A(X \times X)$. However, $Lan_i W^o = \underline{W}$ for any extremally disconnected $W$, hence $Lan_i$ does not preserve equalisers.
\end{prop}
\begin{proof}
The equaliser of the $\pi_i$ between $Y^o$ and $X^o$ can be computed pointwise:
\[
\{f \colon S(\B) \to Y \mid f\text{ open}, \pi_1 \circ \pi_{X \times X} \circ f = \pi_2 \circ \pi_{X \times X} \circ f\}
\]
thus $f$ is a map so that $\pi_{X \times X} \circ f$, which is \qo{}, remains confined to the diagonal of $X \times X$. However, since $X$ has no isolated points, no point of the diagonal is interior, hence the image of $Y$ under $\pi_{X \times X} \circ f$ has empty interior, which is a contradiction. Thus for non-empty $S(\B)$ there are no such maps, and in particular the equaliser is the initial object. On the other hand, in the continuous category,
\[
\{f \colon S(\B) \to Y \mid \pi_1 \circ \pi_{X \times X} \circ f = \pi_2 \circ \pi_{X \times X} \circ f\} = \{f \colon S(\B) \to \pi_{X \times X}^{-1}\Delta\}
\]
which is clearly
\[
= \underline{\pi_{X \times X}^{-1}\Delta}(S(\B)).
\]
To see that $Lan_i W^o \simeq \underline{W}$, note that for any $F \in \pyk$
\[
\pyk(Lan_iW^o, F) \simeq Sh(\sol)(W^o, i^*F) \simeq F(W) \simeq \pyk(\underline{W}, F).
\]
since $W^o$ is a representable sheaf.
\end{proof}

\begin{remark}
As noted in the introduction, $Sh(\sol_{\neg\neg})$ is the double negation subtopos of $Sh(\sol)$. This proof is standard, see the discussion after Proposition \ref{doubleneg} in the appendix for a sketch of the argument.
\end{remark}

\subsection{$Sh(\sol_{\neg\neg}) \simeq \VD$}
The next step is to establish an equivalence between double negation sheaves on $\sol$ and the category $\VD$. 
\begin{theorem} \label{lev_names_thm}
    There is an equivalence of categories between $\mathbf{VD}$ and $Sh(\sol_{\neg\neg}) $. 
\end{theorem}
Throughout this section, for notational convenience, we will forgo mentioning Stone duality and view sheaves as certain \textit{covariant} functors from the category of complete Boolean algebras to the category of sets (for instance writing $F(\B)$ when we mean $F(S(\B))$).
We will also heavily use the characterizations of covers noted in Fact \ref{negnegcovers-fact}.
The main construction is the following. For each sheaf $F$, let $\dot x_F$ be a name for the quotient of triples $(x,\B,I)$ where
    \begin{itemize}
        \item $\B\in V_\kappa$ is a complete Boolean algebra in the sense of $V$;
        \item $x\in F(\B)$;
        \item $I$ is a $V$-generic filter on $\B$
    \end{itemize}
   and $(x,\B,I)\simeq(y,\C,J)$ if and only if there is a complete Boolean algebra $\mathbb{D}$, a generic filter $K$ on $\mathbb{D}$, and ground model complete Boolean homomorphisms $i_1\colon \B\to\D$ and $i_2\colon\C\to\D$ such that $I=i_1^{-1}K, J=i_2^{-1}K$, and $F(i_1)(x)=F(i_2)(y)$. 
   We may think of this as the quotient of some large disjoint union
   \[x_F=\left(\coprod_{\substack{|\B|<\kappa\\I\,\B\text{-generic}}}F(\B)\right)/\simeq.\]
    \begin{lemma}
        $\simeq$ defines an equivalence relation. 
    \end{lemma}
    \begin{proof}
        Reflexivity and symmetry are immediate. 
        For transitivity, suppose some condition $p$ forces that $(x,\B_0,\dot I_0)\simeq(y,\B_1,\dot I_1)\simeq(z,\B_2,\dot I_2)$.
        By extending $p$ if necessary, we may assume there are $\C_0,\C_1$, names for generic filters $\dot J_0,\dot J_1$, and complete embeddings $i_0\colon \B_0\to\C_0$, $i_1\colon\B_1\to\C_0$, $j_0\colon\B_1\to\C_1$, and $j_1\colon\B_2\to\C_1$ which $p$ forces to witness that $(x,\B_0,\dot I_0)\simeq(y,\B_1,\dot I_1)\simeq(z,\B_2,\dot I_2)$. 
        Let $k_0\colon\C_0\to\K_p$ and $k_1\colon\C_1\to\K_p$ be the homomorphisms from Fact \ref{homs_names_fact} corresponding to $\dot J_0$ and $\dot J_1$ respectively.
        Let $\D\subseteq\K$ be the complete subalgebra generated by $\im(k_0)\cup\im(k_1)$. 
        Then $p$ forces that the canonical generic filter intersected with $\D$ together with $k_0\circ i_0$ and $k_1\circ j_1$ witness $(x,\B_0,\dot I_0)\simeq(z,\B_2,\dot I_2)$. 
    \end{proof}
    
    Observe that the formula defining $\dot x_F$ uses only the parameters $F$ and $V_\kappa$. Moreover, $\Vdash_\K\dot x\in V(\R)$ since every generic for every small Boolean algebra in an extension by $\K$ is coded by a real. 
    In particular, $\dot x_F$ is an object in $\Lev$. 
    Given any morphism $f\colon F\to G$, we obtain a name for a function $\dot f\colon\dot x_F\to\dot x_G$ induced by $\dot f(x,\B,I)=(f_{S(\B)}(x),\B,I)$. 
    In this way, we obtain a functor
    \[
    \Phi\colon Sh(\sol_{\neg\neg}) \to \Lev
    \]
    given on objects by $\Phi(F)=\dot x_F$. 

    We will show that this functor is fully faithful and essentially surjective. The following lemma will be useful throughout the proof.
    \begin{lemma} \label{sim_eq_lem}
        Suppose $x,y\in F(\B)$ are such that $\K$ forces that for every $\B$-generic $I$, $(x,\B,I)\simeq (y,\B,I)$. 
        Then $x=y$. 
    \end{lemma}
    \begin{proof}
        Fix $x,y\in F(\B)$ as in the hypotheses of the lemma and fix a complete embedding $\varphi\colon\B\to\K$. 
        By hypothesis, $\Vdash (x,\B,{\varphi^{-1}G})\simeq (y,\B,{\varphi^{-1}G})$ so we may fix a maximal antichain $\langle p_i\mid i<\alpha\rangle$, complete Boolean homomorphisms $k_{i,x},k_{i,y}\colon \B\to\D_i$, and names $\dot J_i$ for $\D_i$-generic filters such that $p_i$ forces that $k_{i,x},k_{i,y},$ and $\dot J_i$ witness $x\simeq y$. 
        Let $j_i\colon\D_i\to\K_{p_i}$ be the homomorphisms induced by $\dot J_i$ as in Fact \ref{homs_names_fact}. Since $p_i$ forces $k_{i,x}^{-1}\dot J_i=k_{i,y}^{-1}\dot J_i$, the composites $j_i\circ k_{i,x}$ and $j_i\circ k_{i,y}$ are both given by $\pi_{p_i}\circ \varphi$. 
        Then $\pi_{p_i}\circ\varphi\colon\B\to j_i[\D_i]$ form a cover of $\B$ and $F(\pi_{p_i}\circ\varphi)(x)=F(\pi_{p_i}\circ \varphi)(y)$ for each $i$. 
        Since $F$ is a sheaf, $x=y$.
    \end{proof}
    
    \begin{lemma} \label{faithful}
        $\Phi$ is faithful.
    \end{lemma}
    \begin{proof}
        Suppose $f,g\colon F\to G$ are distinct morphisms and find $\mathbb{B}$, $x\in F(\mathbb{B})$ such that $f(x)\neq g(x)$. 
        By homogeneity of $\K$ and Lemma \ref{sim_eq_lem}, there is a generic filter $I$ on $\B$ such that $(f(x),\B,I)\not\simeq (g(x),\B,I)$. 
        Then $\Phi(f)([x,\B,I])\neq \Phi(g)([x,\B,I])$. 
        
    \end{proof}
    \begin{lemma} \label{full}
        $\Phi$ is full. 
    \end{lemma}
    \begin{proof}
        
        Suppose that $\dot f\colon\dot x_F\to\dot x_G$ is a morphism in $\Lev$. 
        The key claim is the following:
        \begin{subclaim} \label{key_claim_full}
            For any $\B$ and $x\in F(\B)$, there is a $y\in G(\B)$ such that \[\Vdash_\K\forall I\subseteq\B{\text{ generic, }}y_I\in\dot f([(x,\B,I)]).\]
        \end{subclaim}
        \begin{proof}
        Let $i\colon\B\to\K$ be a complete embedding and fix a maximal antichain $\langle p_i\mid i<\alpha\rangle$ and complete Boolean algebras $\C_i$, names $J_i$ for $\C_i$-generic filters, and $y_i\in G(\C_i)$ such that 
        \[p_i\Vdash(y,\C_i,J_i)\in\dot f([(x,\B,{i^{-1}G})]).\]
        Let $j_i\colon\C_i\to\K_{p_i}$ be the complete Boolean homomorphisms induced by $\dot J_i$. 
        By replacing $\C_i$ with the subalgebra of $\K_p$ generated by $j_i"\C_i\cup\pi_{p_i}"\B$, we may assume that $j_i$ is the identity and that $\C_i$ contains the image of $\B$ under $\pi_{p_i}$. 
        Then $\sum_ij_i\colon\prod_i\C_i\to\K$ defined by $\sum_ij_i(y)=\bigvee y_i$ is a complete embedding from $\prod_i\C_i$ to $\K$ whose image contains the image of $\B$. 
        Moreover, since the projection maps $\{\prod_i\C_i\to\C_i\mid i<\alpha\}$ define a cover such that any complete Boolean homomorphism from $\prod_i\C_i$ to any nonterminal $\D$ factors through at most one $\C_i$ in at most one way, the natural map $ G(\prod_i\C_i)\cong\prod_i G(\C_i)$ is an isomorphism. 
        In particular, there is $y\in G(\prod_i\C_i)$ such that $\Vdash_\K (y,\prod_i\C_i,{\sum_{i}j_i^{-1}G})\in \dot f([x_{i^{-1}G}])$. 
        
        By ultrahomogeneity of $\K$ plus the fact that every $\prod_i\C_i$-generic filter in the extension is the preimage of $G$ under some ground model complete embedding from $\prod_i\C_i$ to $\K$, 
        $\K$ forces that if $J$ is $\prod_i\C_i$ generic and $I=i^{-1}J$, then $(y,\prod_i\C_i,J)\in\dot f([(x,\B,I)])$. 
        We now show that $y$ yields a matching family with respect to the complete embedding $i\colon\B\to\prod_i\C_i$. 
        Since $G$ is a sheaf, we will then find that $y$ is in the image of $G(\B)$, completing the proof of Claim \ref{key_claim_full}. 

        Suppose $k_1,k_2\colon\prod_i\C_i\to\D$ are complete Boolean homomorphisms such that $k_1\circ i=k_2\circ i$. 
        Then $\K$ forces that whenever $K$ is $\D$-generic and $I=(k_1\circ i)^{-1}K$, $G(k_i)(y)\in\dot f([x,\B,I])$. 
        In particular, $\K$ forces that whenever $K$ is $\D$-generic,
        \[
        (G(k_1)(y),\D,K)\simeq(G(k_2)(y),\D,K).
        \]
        By Lemma \ref{sim_eq_lem}, $G(k_1)(y)=G(k_2)(y)$, as desired. 
        \end{proof}
        \subclaimsdone
        We now finish the proof of Lemma \ref{full}. 
        For each $\B$, let $f_\B\colon F(\B)\to G(\B)$ be a function such that for each $x\in F(\B)$, $f_\B(x)$ satisfies the conclusion of Claim \ref{key_claim_full}. 
        To complete the proof of Lemma \ref{full}, it suffices to show that $f_\B$ defines a natural transformation. 
        Let $k\colon\B\to\C$ be a complete Boolean homomorphism and fix $x\in F(\B)$. 
        Then $\K$ forces that if $J$ is a generic filter on $\C$ and $I=k^{-1}J$, then $(f_\C\circ F(k)(x),\C,J)$ and $(G(k)\circ f_\B(x),\C,J)$ are both in the class $\dot f([x,\B,I])$; in particular,
        \[
        (f_\C\circ F(k)(x),\C,J)\simeq(G(k)\circ f_\B(x),\C,J).
        \]
        By Lemma \ref{sim_eq_lem}, $f_\C\circ F(k)(x)=G(k)\circ f_\B(x)$, as desired. 
    \end{proof}

    \begin{remark} \label{sheafification_names_rem}
        The definition of $\dot x_F$ makes sense for arbitrary presheaves $F$. 
        The proofs of Lemmas \ref{faithful} and \ref{full} that $\Phi$ is fully faithful on maps $F \to G$ required only that $G$ be a sheaf. 
        As a consequence, the assignment from a presheaf $F$ to the name $\dot x_F$ can be viewed as the composite $\Phi\circ \a$ with the sheafification functor. Explicitly, for a presheaf $F$ and name $\dot a$,
        \[
        \VD(\Phi \circ \mathbf{a}(F), \dot a) \simeq Sh(\sol_{\neg\neg})(\mathbf a (F), \Phi^{-1}\dot a)
        \]
        \[
        \simeq Psh(\sol_{\neg\neg})(F, \Phi^{-1}\dot a) \simeq \VD(\Phi(F), \dot a)
        \]
        hence $\Phi(F) \simeq \Phi \circ \mathbf{a}(F)$ by the Yoneda lemma.
    \end{remark}

    \noindent Suppose $\dot y$ is an object of $\Lev$. 
    Fix a formula $\varphi(x,a,r)$ for some $a\in V$ and $r\in\R$,
    \[\Vdash_\K\dot y=\{z\mid \varphi(z,a,r)\}.\]
    By the reflection theorem, fix an ordinal $\gamma$ such that 
    \begin{itemize}
        \item $\Vdash_\K\dot y=\{z\mid V_\gamma^{V[G]}\models\varphi(z,a,r)\}$;
        \item $\K$ forces that every element of $y$ is definable inside $V_\gamma^{V[G]}$ using parameters in $V_\gamma^V$ and real numbers. 
    \end{itemize}
    Let $\lambda>\max(\kappa,2^\gamma)$ be regular. 
    We consider the assignment $F$ given by \[F(\B)= \{(\dot z,\psi)\mid\,\Vdash_\K^{H_\lambda}\text{for every }V\text{-generic }I\text{ over }\B,V_\gamma^{V[G]}\models\psi(\dot z(I))\in\dot y\}/\equiv,\]
    where $\dot z$ is a $\B$-name, $\psi$ is a formula with parameters from $V_\gamma$, and 
    \[
    (\dot z_1,\psi_1) \equiv (\dot z_2,\psi_2) \text{\ \ iff\ \ }\Vdash_\K^{H_\lambda} \forall I(\psi_1(\dot z_1(I))=\psi_2(\dot z_2(I))).
    \]
    If $f\colon\B\to\C$ is a complete Boolean homomorphism, we obtain a map from $\B$-names to $\C$-names as in Section \ref{auts_subsection} and define \[F(f)(\dot y,\psi)=(f(\dot y),\psi).\] 
    \begin{lemma}
    $F$ defines a sheaf. 
    \end{lemma}
    \begin{proof}
        Suppose $\{S(f_i)\colon S(\C_i)\to S(\B)\mid i<\alpha\}$ is a cover where $f_i \colon \B \to \C_i$ are the dual complete homomorphisms. We must show that every matching family $(\dot z_i,\psi_i)$ is given by $F(f_i)(x)$ for some unique $x\in F(\B)$. 
        %We let $\Theta$ be the formula asserting that $z$ is the 
        For existence, we show the following:
        \begin{subclaim}
            Let $\dot G$ be the canonical $\B$-name for a $V$-generic filter over $\B$. Let $\psi$ be the definable class function for which
            \[
            \forces_\K \psi(\dot G) = \dot z
            \]
            iff
            \[
            \forces_\K \forall i < \alpha \forall I \subseteq \C_i\text{ generic, }(f_i^{-1}I = \dot G \r \psi_i(\dot z_i(I)) = z).
            \]
            %\[(\dot G, ``\text{the unique }z\text{ such that }\forall i<\alpha\forall I\subseteq\C_i \text{ which are }V-\text{generic and satisfy }f_i^{-1}I=\dot G(\psi_i(\dot z_i(I))\text{ defines }z\text{ over }V_\gamma)")\]
            Then $(\dot G, \psi)$ defines a member of $F(\B)$. 
        \end{subclaim}
        \begin{proof}
            Since every generic filter on $\B$ is the preimage of a generic filter on $\C_i$ for some $i$, it is enough to show that if $i,j<\alpha$, $p\in\K$ and $H,H'$ are names for $\C_i,\C_j$ generic filters such that $p\Vdash f_i^{-1}H=f_j^{-1}H'$, then $p$ forces that $\psi_i(\dot z_i(H))$ and $\psi_j(\dot z_j(H'))$ define the same member of $\dot y$. 
            To this end, let $h\colon\C_i\to\K_p$ and $h'\colon\C_i'\to\K_p$ be the complete Boolean homomorphisms induced by $H,H'$ and $\D\subseteq\K_p$ be the complete algebra generated by $\im(h)\cup\im(h')$. 
            Since $p$ forces that $f_{i}^{-1}H=f_{j}^{-1}H'$, we find $h\circ f_i=h'\circ f_j$. 
            Since $(\dot z_k,\psi_k)$ are a matching family, $F(h)(z_i,\psi_i)=F(h')(z_j,\psi_j)$; let $(\dot z,\theta)$ be a member of this common equivalence class. 
            By definition of $F$, $p$ forces that $\psi(\dot y_i(H))$ and $\psi(\dot y_i(H'))$ both define the same element of $\dot y$ as $\theta(\dot z(G\cap\D))$, as desired. 
        \end{proof}
        \subclaimsdone
        For uniqueness, suppose that $F(f_i)(\dot y,\varphi)=F(f_i)(\dot z,\psi)$ for each $i<\alpha$. Since every $\B$ generic filter is the preimage of a $\C_i$ generic filter for some $i$, $\K$ forces that whenever $I$ is $\B$ generic, $\varphi(\dot y(I))\equiv \psi(\dot z(I))$. 
        In particular, $(\dot y,\varphi)\equiv(\dot z,\psi)$.
    \end{proof}
    
    \begin{lemma} \label{ess_surj_lem}
        $\Phi$ is essentially surjective. 
    \end{lemma}
    \begin{proof}
    
    With $F$ as above, we obtain a morphism $\dot f$ in $\Lev$ from $\dot x_F$ to $\dot y$ as a name which assigns $((\dot z,\psi),\B,I)$ to $\psi(\dot z(I))$. Thus it suffices to see that $\dot f$ is forced to be a bijection (and therefore is an isomorphism in $\Lev$).
    
        To see that $\dot f$ is forced to be an injection, suppose that for some complete Boolean algebras $\B,\C$, some $x\in F(\B)$ and $y\in F(\C)$, names $\dot I,\dot J$ for $\B$ and $\C$ generic filters, and condition $p$,
        \[p\Vdash \dot f(x,\B,I)=\dot f(y,\C,J).\]
        Let $i\colon\B\to\K_p$ and $j\colon\C\to\K_p$ be the complete Boolean homomorphisms induced by $\dot I$ and $\dot J$ respectively.
        Let $\D$ be the complete subalgebra of $\K_p$ generated by $i"\B\cup j"\C$.
        By choice of $p$ and the definition of $\equiv$, $F(i)(x)=F(j)(y)$.
        Moreover, $p$ forces that $i^{-1}(G\cap\D)=I$, $j^{-1}(G\cap \D)=J$. 
        In particular, $\K$ forces that $(x,\B,{\dot I})\simeq (y,\C,{\dot J})$, as desired. 

        To see that $\dot f$ is forced to be a surjection, suppose that in some generic extension $V[G]$, $z\in \dot y$. 
        By choice of $\gamma$, there is a complete Boolean algebra $\B$, a $\B$-name $\dot w$, and a $\B$-generic filter $I$ such that $\psi(\dot w(I))$ defines $\dot y$. 
        Since $\dot x$ is definable from ground model parameters and $\K$ is homogenous, there is a $q\in I$ such that 
        \[q\Vdash_\B\Vdash_{\K}\psi(\check{\dot w})\text{ defines a member of }\dot y\text{ in the structure }V_\gamma^{V[G]}.\]
        Then $z=\dot f((\dot y,\psi),\B_q,{I\cap\B_q})$.

    \end{proof}

This shows that $\Phi$ is fully faithful and essentially surjective, and hence completes the proof of Theorem \ref{lev_names_thm}.

\begin{remark}
\label{r_star_remark}
A similar result to Theorem \ref{lev_names_thm} holds if $\kappa$ is assumed to be merely a strong limit cardinal, but we must replace $\R$ in Definition \ref{Lev_def} with $\R^*$, the reals added at some proper initial stage.
\end{remark}

\begin{remark}
An earlier version of this proof passed through the topos
\[
Sh_{\mathbf{B}Aut(\K)}(\K),
\]
i.e. canonical sheaves on $\K$ internal to the continuous $Aut(\K)$-sets. This topos, due to work of Fourman \cite{fourman} and Hayashi \cite{hayashi}, corresponds to the symmetric extension by the system
\[
(Aut(\K), \{Fix(\B) \mid \B \subseteq \K \text{ small}\}, \K),
\]
which is also the model $V(\R)$. The proof proceeded by covering $Sh(\sol_{\neg\neg})$ with this topos, and using the latter's in-built structure as a category of ``$Aut(\K)$-equivariant $\K$-sheaves" to produce names for $V(\R)$ sets. This approach directly parallels that of Grigorieff \cite{grigorieff} and is how we initially obtained the definition for $\VD$.

However, this came with a big caveat: $Sh_{\mathbf{B}Aut(\K)}(\K)$ is \textbf{not} a well-founded topos, i.e. contains objects that do not embed in its cumulative hierarchy, and hence do not correspond to names (the interested reader may want to check this by looking at the representable sheaves%; see also Fourman \cite{fourman} for how such a situation arises in general when representing models of set theory with atoms
). Nevertheless, the objects of this topos expressible as names in $V(\R)$ are all and exactly those that it inherits from $Sh(\sol_{\neg\neg})$! In other words, while $Sh_{\mathbf{B}Aut(\K)}(\K)$ is the ``canonical" way to interface between symmetric extensions and topoi, arguably $Sh(\sol_{\neg\neg})$ hews more closely to the symmetric names. This hints at the possibility that, on the set theoretic side, there is the potential to gain more information from a symmetric extension by identifying all the hidden maps (outside of the action of $Aut(\K)$) that act on the class of names. 
\end{remark}

Finally we are in a position to define the transfer functor, which is the left component of the geometric morphism defined by the previous two sections.

\begin{defn}[the $\Lambda$ functor]
We denote by $\Lambda$ the composite
\[
\pyk \xrightarrow{\ i^* \ } Sh(\sol) \xrightarrow{\ \mathbf{a}\ } Sh(\sol_{\neg\neg}) \xrightarrow{\ \Phi\ } \VD.
\]
where $\mathbf{a}$ denotes sheafification. We note for clarity that this is left adjoint to
\[
\VD \xrightarrow{\ \Phi^{-1} \ } Sh(\sol_{\neg\neg}) \hookrightarrow  Sh(\sol) \xrightarrow{\ Ran_i\ } \pyk.
\]
\end{defn}

\subsection{The Images of Some Pyknotic Sets}
We now turn to computing the image of some condensed sets under $\Lambda$. 
\subsubsection{Weak Hausdorff spaces}
We begin with a general computation of the image of representables from condensed sets. 
Perhaps remarkably, the image of a representable coincides precisely with the interpretation used in the analysis of \blhs. 
Recall that a space $X$ is \emph{weak Hausdorff} if whenever $K$ is a compact Hausdorff space and $f\colon K\to X$ is continuous, the image $f(K)$ is Hausdorff in the subspace topology.
\begin{defn}
For $X$ weak Hausdorff, we write $K(X)$ to mean the poset of nonempty compact Hausdorff subsets of $X$, ordered by containment.
\end{defn}
Note that $K(X)^V$ can be equipped with a natural topology given by declaring $\{F\mid C\in F\}$ to be closed for $C\in K(X)^V$. However, in some sense this has nothing to do with $\Lambda$, as the notion of isomorphism visible to $\VD$ is that of definable \textit{bijection} and not homeomorphism. Given the algebraic facts that are nevertheless preserved (see section \ref{whitehead_section}), we hold it as quite surprising that in the presence of sufficient failures of choice, this weak notion of bijection is able to ``mimic" homeomorphism.

\begin{prop} \label{reps_names_prop}
    If $X$ is a weak Hausdorff space, then $\Lambda(\underline{X})$ is naturally isomorphic to the set of maximal filters on $K(X)^V$ which are in the Solovay model $V(\R)$.
\end{prop}
\begin{proof}
    As noted in Remark \ref{sheafification_names_rem}, the composition $\Phi\circ \a(F)$ of a presheaf $F$ on $\sol$ is, up to natural isomorphism, the name $\dot x_F$. 
    In particular, $\Lambda(\underline{X})$ can be identified with a name for the quotient of the set of triples $(f,\B,I)$ for 
    \begin{itemize}
        \item $\B\in V_\kappa$ a complete Boolean algebra,
        \item $f\colon S(\B)\to X$,
        \item and $I$ a $V$-generic filter on $\B$,
    \end{itemize}
    by the relation $(f,\B,I)\simeq(g,\C,J)$ if and only if there is a complete Boolean algebra $\D$, generic filter $K$ on $\D$, complete Boolean homomorphisms $i\colon\B\to\D$ and $j\colon \C\to\D$ such that $I=i^{-1}K,J=j^{-1}K$, and $f\circ S(i)=g\circ S(j)$. 

    Given a generic filter $I$ on $\B$ and $f\in C(S(\B),X)$, we let the filter $\F_{f,I}$ be $\{K\in K(X)\mid\exists p\in I(N_p\subseteq \inv{f}K)\}$ where $N_p$ is the basic clopen set of all ultrafilters containing $p$. 
    \begin{subclaim}
        $\F_{f,I}$ is a maximal filter on $K(X)$.
    \end{subclaim}
    \begin{proof}
        $\F_{f,I}$ is nonempty since it contains $f(S(\B))$. 
        $\F_{f,I}$ is closed under finite intersections since if $N_p\subseteq \inv{f}K$ and $N_q\subseteq \inv{f}K'$ then $N_{p\wedge q}\subseteq \inv{f}(K\cap K')$. 
        Maximality of $\F_{f,I}$ follows from genericity of $I$ since any $K$, $\{p\mid N_p\subseteq \inv{f}K\text{ or } N_p\cap \inv{f}K=\varnothing\}$ is dense as $\inv{f}K$ is closed. 
    \end{proof}
    \noindent To complete the proof of Proposition \ref{reps_names_prop}, it suffices to show the following:
    \begin{subclaim}
        The assignment $\F$ which assigns $f_I$ to $\F_{f,I}$ induces a bijection between equivalence classes of functions and maximal filters on $K(X)$. 
    \end{subclaim}
    \begin{proof}
        We first show that $\F$ is well-defined on equivalence classes. 
        Suppose that for some $f\colon S(\B)\to X$ and $g\colon S(\C)\to X$ and generic filters $I,J$ on $\B,\C$ that $(f,\B,I)\simeq(g,\C,J)$, as witnessed by $\D,K,i,j$. Let $C\in \F_{f,I}$ and fix $p\in I$ such that $N_p\subseteq \inv{f}C$. Then $i(p)\in K$ and $N_{\varphi(p)}\subseteq (f\circ S(i))^{-1}K$. 
        Since $j$ is a complete homomorphism, the image $s(j)[N_{\varphi(p)}]$ is open in $S(\C)$ so that \[\{q\in \C\mid N_q\subseteq S(j)[N_{i(p)}]\text{ or }N_q\cap S(j)[N_{i(p)}]=\varnothing\}\]
        is dense in $\C$. 
        Since $\psi^{-1}K=J$, the second alternative cannot occur for any condition in $J$, so there is $q\in J$ such that $N_q\subseteq S(j)[N_{\varphi(p)}]$. 
        Since $f\circ S(i)=g\circ S(j)\colon S(\D)\to X$, $N_q\subseteq g^{-1}K$ so $C\in\F_{g,J}$, as desired. 
        The reverse containment holds by symmetry. 

        We next show that $\F$ is injective on equivalence classes. Suppose that for some condition $p$, complete Boolean algebras $\B,\C$, continuous $f\colon S(\B)\to X$ and $g\colon S(\C)\to X$, and names for generic filters $I,J$ on $\B,\C$ respectively that 
        \[p\Vdash\F_{f,I}=\F_{g,J}.\]
        Let $i\colon\B\to\K_p$ and $j\colon\C\to\K_p$ be the complete Boolean homomorphisms induced by $I,J$ respectively and let $\D\subset\K_p$ be the complete subalgebra generated by $\im(i)\cup\im(j)$. 
        Then the computation in \cite[Theorem 8.11]{WPCM}, $f\circ S(i)=g\circ S(j)$ is the function which assigns $U\in S(\D)$ to the unique point in
        \[\bigcap\{C'\mid\llbracket C'\in \F_{f,i^{-1}G}\rrbracket\in U\}\] and $p$ forces that $G\cap\D$ witnesses $f_I\simeq g_J$. 

        Finally, we show that $\F$ is surjective. First, any maximal filter $\F$ on $K(X)^V$ which is in the Solovay model is definable from parameters added by some generic $I$ on a $\kappa$-small $\B$. 
        By homogeneity of the tail forcing over $V[I]$, $\F\in V[I]$. 
        Let $\dot F$ be a $\B$-name for a maximal filter on $K(X)^V$ such that $\dot F(I)=\F$ and fix $p\in\B$ such that for some $C\in K(X)^V$, $p$ forces $C\in\F$. 
        Since $C$ is Hausdorff in the subspace topology, \cite[Theorem 8.11]{WPCM} yields the desired continuous function. 
        
    \end{proof}
    \subclaimsdone
\end{proof}

\subsubsection{Compact Hausdorff Spaces}

We note here some remarks on the special case of Theorem \ref{reps_names_prop} for $X$ a compact Hausdorff space. 
Perhaps the most immediate corollary is that Stone spaces map to the ultrafilters on the corresponding Boolean algebra. 
\begin{cor} \label{Stone_images}
    If $\B$ is a small Boolean algebra, $\Lambda(\underline{S(\B)})$ is naturally isomorphic to the set of ultrafilters on $\B$. 
\end{cor}
\noindent A routine verification shows that $\Lambda(\underline{[0,1]^I})\cong[0,1]^I$. More generally, we can obtain 
\begin{cor} \label{compact_images}
    Suppose $X\subseteq[0,1]^I$ is closed. $\Lambda(\underline{X})$ is isomorphic to the underlying set of $\overline{\check{X}}\subseteq[0,1]^I$.
\end{cor}
As natural as Corollaries \ref{Stone_images} and \ref{compact_images} may seem, they have some consequences that feel less natural. 
The main cause of this is that all atomless Boolean algebras of size ${<}\kappa$ in $V$ become isomorphic in $V(\R)$, with Stone spaces homeomorphic to the Cantor space. 
Although the homeomorphism is sometimes not in $\VD$, the Stone spaces consequently lose many topological properties.

In certain cases, there are other ways to map a compact space into a forcing extension that do not turn all Stone spaces into the Cantor space; we point interested readers to \cite{CSFBC}, where Todorcevic develops machinery for mapping Rosenthal compacta into forcing extensions which behaves more naturally for this class of spaces. 
See also \cite[Section 12]{MoF} for another account of these results. 
We compute the images of the spaces from \cite[Example 12.1]{MoF} under $\Lambda$ to demonstrate the difference. 
\begin{defn}
    \textbf{Helly's space} $H$ is the collection of monotonically increasing functions $f\colon[0,1]\to[0,1]$, with topology inherited from the product of closed intervals $\prod_{i\in[0,1]}[0,1]$. 
    $H$ is convex as a subspace of $\R^{[0,1]}$; the extreme points in the Helly space are the characteristic functions of intervals $\chi_{(r,1]}$ and $\chi_{[r,1]}$ for $r\in[0,1]$, forming a space homeomorphic to the \textbf{double arrow space} $D$: the lexicographic product $[0,1]\times\{0,1\}$ equipped with the order topology. 
\end{defn}
Proposition \ref{compact_images} allows us to easily compute the images of Helly's space and the double arrow space under $\Lambda$. 
In both cases, we see the defining behavior of the spaces is restricted to only the reals from the ground model: functions in the image of Helly's space need only be defined on reals from the ground model and only reals in the ground model are duplicated in the double arrow space. 
\begin{cor}
    $\Lambda(\underline{H})$ is isomorphic to the collection of monotonically increasing $f\colon[0,1]\cap V\to[0,1]$. 
    $\Lambda(\underline{D})$ is isomorphic to the subset of $\Lambda(\underline{H})$ consisting of 
    \[\{\chi_{(r,1]}\mid r\in [0,1]\}\cup\{\chi_{[r,1]}\mid r\in [0,1]\cap V\}\]
    (note that $r$ is not restricted to $V$ in the first part of the union).
\end{cor}

%\subsubsection{Free Abelian Groups}

\subsubsection{Large Baer-Specker Groups}
We now compute the values of $\Lambda(\prod_I\underline{\Z})$ for each index set $I$. 
\begin{defn}
    We write $\bprod_I\Z$ for the subgroup of $\prod_I \Z$ consisting of all $\sigma$ such that for some $f\colon I\to\omega$ in $V$, $|\sigma(i)|\leq f(i)$ for all $i\in I$. 
\end{defn}
\begin{cor} \label{image_products_cor}
    In the category $\VD$, there is an isomorphism $\Lambda\left(\prod_I\underline\Z\right)\cong\bprod_I\Z.$
\end{cor}
\begin{proof}
    By Proposition \ref{reps_names_prop}, $\Lambda(\prod_I\underline{\Z})$ is isomorphic to the collection of maximal filters on the ground model compact subsets of $\prod_I\Z$. 
    By compactness and continuity of each projection, whenever $K\subseteq \prod_I\Z$ is compact, there is an $f\colon I\to\omega$ such that for all $i\in I$ and all $x\in K$, $|x(i)|\leq f(i)$. A routine check shows that for each $f\colon I\to\omega$, the maximal filters concentrating on 
    \[\left(\prod_{i\in I}[-f(i),f(i)]\right)^V\]
    can be identified with the analogous product $\prod_{i\in I}[-f(i),f(i)]$ in $V(\R)$.
\end{proof}
\pagebreak
\section{Discrete Whitehead Groups are Free} \label{whitehead_section}

\subsection{Homs Between Bounded Products}

The bounded product $\bprod_I \Z$ is central to our transference of Whiteheadness into the Solovay model. We recall the following theorem, due to Blass:

\begin{fact}[Blass \cite{STNG}]
Denote by $\prod^{ubdd}_\omega \Z$ the space of uniformly bounded sequences:
\[
\{\sigma \colon \omega \to \Z \mid \sigma \in V(\R), \exists n \forall m (|\sigma(m)| \leq n)\}.
\]
Suppose a homomorphisms $\pi \colon \prod^{ubdd}_\omega \Z \to \Z$ has the property that the restriction of $\pi$ to $\{0, 1\}^\omega$ has the property of Baire (equivalently, $\pi^{-1}(n) \cap \{0, 1\}^\omega$ has the Baire property for every $n \in \Z$), then there is a finite $F \subseteq I$ so that $\pi$ factors through the truncation to $\prod_F \Z$.
\end{fact}

The ``uniformly bounded" product defined above is in some sense immaterial to the proof. What is needed is the Baire measurability of the restriction to $\{0, 1\}^\omega$, and this is always true (see Fact \ref{measurable_fact}) in the Solovay model. Thus we may extend this property to our groups $\bprod_I \Z$ for small $I$:

\begin{lemma} \label{automatic_continuity_lem}
    Suppose $|I|<\kappa$ and $\pi\colon\bprod_I\Z\to\Z$ is a homomorphism in the Solovay model. 
    There is a finite $F\subseteq I$ such that $\pi$ factors through $\prod_F\Z$. 
    In particular, $\pi$ is continuous. 
\end{lemma}
\begin{proof}
    Since all sets of reals in the Solovay model have the Baire property, and $|I| = \aleph_0$, by the above fact, there is a finite $F\subseteq I$ such that the restriction of $\pi$ to the collection of uniformly bounded sequences factors through $\prod_F\Z$.
    We show $\pi$ factors through $\prod_F\Z$. 
    Suppose that for some $f,g\in \prod_I\Z$, $f\upharpoonright F=g\upharpoonright F$ but $\pi(f)\neq\pi(g)$. 
    Let $N$ be sufficiently large that $\pi(f)\not\equiv\pi(g)\mod N$. 
    Then $\pi$ descends to a homomorphism
    \[\overline{\pi}\colon{\bprod_I\Z} {\Big/}N\bprod_I \Z \simeq \prod_I \Z/N\Z\to \Z/N\Z\]
    such that $\overline{\pi}([f])\neq\overline{\pi}([g])$. 
    But $[f]$ and $[g]$ both have representatives $f',g'$ uniformly bounded by $N$ such that $f'\upharpoonright F=g'\upharpoonright F$, contradicting that $\pi$ factors through $\prod_F\Z$ on the uniformly bounded sequences. 
\end{proof}

The above proof shows that on the objects $\prod_I \Z$, the functor $\Lambda$ is actually fully faithful. As a result, the Clausen-Scholze proof of the freeness of $A$ can actually be transferred to $V(\R)$ wholesale. However, as a proof of concept, we will essentially mimic the strategy of proof in the Solovay model, showing that the proof can actually be carried out entirely set theoretically therein, and using the fullness proved above, be carried back to $\pyk$. To start, we will introduce the following shorthand.

\begin{defn}
Let $f : I \to \omega$ be a function in $V$. We denote by $\prod(f)$ the set of $f$-bounded sequences in $V(\R)$
\[
\prod_I(f) = \{\sigma : I \to \Z \mid \sigma \in V(\R), \forall i \in I (|\sigma(i)| \leq f(i))\}
\]
When $I$ is understood, we will drop it.
\end{defn}

These are exactly the $V$-definable compact subsets of $\prod^{bdd}_I \Z$. Next, we note that if $\pi$ is surjective, this surjectivity can be witnessed ``locally" in the sense of the sets $\prod(f)$.

\begin{lemma} \label{bdd_image_lem}
    Suppose $\pi\colon\bprod_I\Z\to\bprod_J \Z$ is a surjective homomorphism in the Solovay model with $|I|,|J|<\kappa$. 
    There is an $f\colon I\to\omega$ in $V$ such that $\{0,1\}^J\subseteq\pi\left[\prod(f)\right]$. 
\end{lemma}
\begin{proof}
    Suppose not. 
    \begin{subclaim}
        For each $f\colon I\to\omega$ in $V$, $\pi\left[\prod(f)\right]\cap\{0,1\}^J$ is nowhere dense in $\{0,1\}^J$. 
    \end{subclaim}
    \begin{proof}
        Fix a finite partial function $p\colon J\rightharpoonup\mathbb{Z}$ and $f\colon I\to\omega$. For each $j\in\dom(p)$, let $g_j\in\bprod_I\Z$ be such that $\pi(g_j)=e_j$. 
        Let $f_j\colon I\to\omega$ be functions in $V$ such that for all $i\in I$, $|g(i)|\leq f(i)$. 
        By hypothesis, there is an $x\in\{0,1\}^J$ satisfying
        \[x\not\in \pi\left[\prod\left(f+\sum_jf_j\right)\right].\]
        Then 
        \[y = x-\sum_{x(j)\neq p(j)}(x(j)-p(j))e_j\]
        is a member of $N_{p}\setminus\pi[\prod(f)]$.  
        Since $\prod(f)$ is compact and $\pi$ is continuous by Lemma \ref{automatic_continuity_lem}, $N_p\setminus \pi[\prod(f)]$ is a nonempty open subset of $N_p$ disjoint from $\pi[\prod(f)]$, as desired.
    \end{proof}
    Now, since $\omega^I\cap V$ is countable and $\{0,1\}^J$ is Polish in $V(\R)$, the Baire Category Theorem implies there is $x\in\{0,1\}^J\setminus\bigcup_f\pi\left[\prod(f)\right]$. This contradicts the surjectivity of $\pi$.
    
    \subclaimsdone
\end{proof}

The property of $\prod^{bdd}_I \Z$ being generated by a filtered union of compact sets is crucial, as it sets up definitions that are naturally absolute to all small forcing extensions. An example of this phenomenon is the following, which will be used in our proof.

\begin{lemma}
\label{covering_relative_down}
Assume the notation of Lemma \ref{bdd_image_lem}. If
\[
\{0, 1\}^J \subseteq \pi\left[\prod(f)\right]
\]
holds in the Solovay model, then it holds in every intermediate forcing extension $V^\B$ by a small complete Boolean algebra $\B$ (note that the definition makes sense independently of the Solovay model).
\end{lemma}
\begin{proof}
Fix $s\in\{0,1\}^J$. 
For each finite $F\subseteq J$, 
\[A_F=\left\{t\in\prod(f)\mid\pi(t)\upharpoonright F=s\upharpoonright F\right\}\]
is clopen since $\pi$ is continuous. 
Since $s$ is in the image of $\prod(f)$ in the Solovay model and $\prod(f)\cap V$ is already dense in $\prod(f)$, $\{A_F\}_F$ has the finite intersection property in $V^\B$. 
Since $\prod(f)$ is compact in $V^\B$, there is a $t\in\bigcap_{F}A_F\cap V^\B$. 
Then $\pi(t)=s$, as desired. 
\end{proof}

\subsection{Putting it All Together}
\label{whiteheadargumentfinal}

\begin{theorem} \label{Whitehead_implies_free_cor}
    If $A \in V$ with $|A| < \kappa$ is a discrete abelian group, and $\underline{A}$ is Whitehead in $\pyk$, then $A$ is free (in $V$).
\end{theorem}
% \begin{proof}
%     By \cite[Theorem 6.3]{CMSM}, if $\lambda$ is the least size of a nonfree subgroup of $A$ then after adding $\lambda$ many Cohen reals, $A$ is not $W^*$. 
%     Then Corollary \ref{Whitehead_implies_free_cor} follows from Proposition \ref{cond_whitehead_desc_prop} and Theorem \ref{whitehead_thm}.
% \end{proof}

\begin{proof}
Consider a resolution of $A$ in $V$:
\[
0 \to \bigoplus_I \Z \to \bigoplus_J \Z \to A \to 0.
\]
By Whiteheadness in $\pyk$, its dual is exact, and hence applying the functor $\Lambda$ leaves it exact in $\VD$:
\[
0 \leftarrow \prod^{bdd}_I \Z \xleftarrow{\ \pi\ } \prod^{bdd}_J \Z \leftarrow \Lambda (\underline{A}) \leftarrow 0
\]
Note that the left most arrow, which we denote by $\pi$, is surjective by exactness.

Now, by enlarging $I$ and $J$ if necessary (by padding with a free summand), we assume there is an \textbf{extremally disconnected} $S$ so that in $V$,
\[
C(S, \Z) \simeq \bigoplus_I \Z.
\]
Note that by N\"obeling's theorem \cite{Nobeling}, $C(S, \Z)$ is free for any compact Hausdorff $S$ whatsoever, i.e. the above is only a size constraint on $I$ since we have insisted $S$ be extremally disconnected. We note, following the proof of N\"obeling's theorem, that we can furthermore identify $I$ with a family $\{p_\alpha \mid \alpha \in I\}$ of clopen subsets of $S$, and that the basis vectors of $\bigoplus_I \Z$ correspond to the indicator functions $\chi_{p_\alpha} \in C(S, \Z)$.

This in turn produces a map $g \colon S \to \{0, 1\}^I$ via
\[
x \mapsto (\chi_{p_\alpha}(x))_\alpha
\]
which is manifestly continuous (hence a closed immersion). Note injectivity is a consequence of the fact that $C(S, \Z)$ separates points (i.e. for any two points, there is an element of $C(S, \Z)$ which distinguishes them, hence not all elements of the basis $\chi_{p_\alpha}$ can coincide on them). If $S = S(\B)$, this map corresponds to a $\B$-name $\sigma$ with the property that
\[
\forces_\B \sigma \in \{0, 1\}^I.
\]
Since $\pi$ is surjective in $V(\R)$, by Lemma \ref{bdd_image_lem} there is $f \colon J \to \omega$ in $V$ so that $\{0, 1\}^I \subseteq \pi[\prod(f)]$. Furthermore by Lemma \ref{covering_relative_down}, we have also
\[
\forces_\B \{0, 1\}^I \subseteq \pi[\prod(\check{f})]
\]
whence by maximality (where we use that $\B$ is complete, or equivalently that $S$ is projective) there is a $\B$-name $\tau$,
\[
\forces_\B \tau \in \prod(\check{f}), \pi(\tau) = \sigma.
\]
In turn, $\tau$ corresponds to a continuous function $h \colon S \to \prod(f)$ in $V$, with the property that $\pi \circ h = g$. We claim this $h$ induces a splitting of $\pi$ in $V$, which will complete the proof.

Since $g$ is a closed immersion, so is $h$. Since the coordinates of $h \colon S \to \prod_J \Z$ are continuous maps they are finite linear combinations of the $\chi_{p_\alpha}$. These finite linear combinations define a matrix, more concretely for $x \in S$,
\[
h(x) = \Big(\sum_{\beta \in I} c_{\alpha\beta} \chi_{p_\beta}(x)\Big)_{\alpha \in J}
\]
where all sums are finite (and the coefficients are constant since $\chi_{p_\alpha}$ is a basis). Hence $h$ defines a row-finite $J \times I$ matrix $h'$ (defined so that $h' \circ g = h$), again in $V$, and hence already defines a map
\[
\bprod_I \Z \to \bprod_J \Z
\]
so that on the image of $S$, $\pi \circ h' = id$. Now we note that
\begin{subclaim}
If two row-finite matrices are equal on the image of $S$, then they are equal everywhere.
\end{subclaim}
\begin{proof}
Going coordinate-by-coordinate, and subtracting the one matrix from the other, we are reduced to showing that if for some finite support vector $(c_\alpha)_\alpha$, we have
\[
\sum_\alpha c_\alpha \chi_{p_\alpha}(x) = 0
\]
for every $x \in S$, then the $c_\alpha$ are zero. But this is immediate, since the $\chi_{p_\alpha}$ are linearly independent in $C(S, \Z)$.
\end{proof}
\subclaimsdone

Hence $\pi \circ h' = id$ everywhere, in particular this sequence is split exact. Hence its $V$-definable dual (i.e. applying external $\mathbf{VD}$-homomorphisms $\Abvd(\cdot, \Z)$ to the sequence) is split exact. But by Lemma \ref{automatic_continuity_lem}, $\Abvd(\bprod_I \Z, \Z) \simeq \bigoplus_I \Z$, hence
\[
0 \to \bigoplus_I \Z \xrightarrow{\pi^*} \bigoplus_J \Z \to \Abvd(\Lambda(\underline{A}), \Z) \to 0
\]
is split exact. But $\pi^*$ is the same map as in the original resolution, hence we get
\[
\Abvd(\Lambda(\underline{A}), \Z) \simeq A.
\]
This is now the original resolution of $A$, except that it is split by the dual of $h'$. As noted, $h'$ and its dual live already in $V$. Hence $V$ has already a splitting for the resolution of $A$. Hence $A$ was free to start with.
\end{proof}

\begin{remark}
It must be stressed that aside from Lemmas \ref{automatic_continuity_lem} and \ref{bdd_image_lem}, which are properties specific to the Solovay model, the above proof quite closely replicates exactly the same manipulations that lead to the freeness of $A$ in the proof by Clausen--Scholze \cite[Session 8]{clausen_lecture}.
\end{remark}

As noted in \blhs \cite{WPCM}, this result extends to all abelian groups in $V$ (Clausen--Scholze's conclusion) under the assumption of $\pyk$ being ``sufficiently close to $\mathbf{Cond}$", namely when $\kappa$ is strongly compact:

\begin{cor}[{\blhs \cite[Corollary 11.4]{WPCM}}]
If $\kappa$ is strongly compact, and $A$ is Whitehead in $\pyk$, then $A$ is free.
\end{cor}
\begin{proof}
We sketch the proof for completeness. If $A$ is Whitehead in $\pyk$, so are all of its subgroups of size $< \kappa$. Hence its subgroups of size $< \kappa$ are free by Theorem \ref{Whitehead_implies_free_cor} above. Since $\kappa$ is strongly compact, Eklof and Mekler show in \cite{AFM} that an abelian group, all of whose subgroups of size $< \kappa$ are free, must itself be free.
\end{proof}
\pagebreak
\section{Extensions in $\VD$} \label{ext_section}

We will in this section show that the first internal derived hom functor between $\R$ and $\Z$ vanishes in $\VD$. However, we must caveat this result by noting that we are not at present aware of a way to derive the vanishing of $\underline{Ext}(\underline{\R}, \underline{\Z}) = R\underline{Hom}^1(\underline{\R}, \underline{\Z})$ in $\pyk$ from it.

We denote by $\dashhom$ the internal group of morphisms in $\Abvd$, in analogy with the internal hom functor $\underline{Hom}$ of $\mathbf{Ab}_\pyk$.
\begin{theorem}
\label{extthm}
Every exact sequence
\[
0 \to \Z \hookrightarrow A \xrightarrow{\pi} \R \to 0
\]
in $V(\R)$ splits.
\end{theorem}
The internal semantics can be used to formalise that the above suffices for the vanishing of $R\dashhom^1(\R, \Z)$ in $\VD$. We will sketch why this is the case (an interested reader may find a more thorough exposition in the Stacks Project Chapter 21 \cite[\href{https://stacks.math.columbia.edu/tag/01FQ}{Tag 01FQ}]{stacks}, more specifically \cite[\href{https://stacks.math.columbia.edu/tag/08J7}{Tag 08J7}]{stacks}). Note that $\VD$ is a Grothendieck topos, hence its category of abelian groups permits $K$-injective resolutions (see \cite[\href{https://stacks.math.columbia.edu/tag/079P}{Tag 079P}]{stacks}). In turn this means that the internal $R\dashhom^1(\R, \Z)$ can be computed by taking internal homs, i.e. it is the homology of the sequence
\[
0 \to \dashhom(\R, I_1) \to \dashhom(\R, I_2) \to \dashhom(\R, I_3) \to \cdots
\]
for any $K$-injective resolution
\[
0 \to \Z \xrightarrow{\partial_0} I_1 \xrightarrow{\partial_1} I_2 \xrightarrow{\partial_2} I_3 \to \cdots.
\]
The internal homs $\dashhom(\R, I_n)$ are simply sets of not-necessary-$V$-definable abelian group homomorphisms in $V(\R)$, hence exactness at $\dashhom(\R, I_2)$ (the vanishing of $R\dashhom^1(\R, \Z)$) means that every homomorphism $h : \R \to I_2$ in $V(\R)$ for which $\R \xrightarrow{h} I_2 \xrightarrow{\partial_2} I_3$ is zero must lift to $I_1$.

Letting $h \colon \R \to I_2 \in V(\R)$ be such a map, we reason classically that $h$ factors through the image of $\partial_1 \colon I_1 \to I_2$ by exactness of the resolution. From this we get the pulled back exact sequence:
\[
\begin{tikzcd}
	0 & \Z & {I_1 \times_{\im \partial_1} \R} & \R & 0 \\
	\\
	0 & \Z & {I_1} & {\im \partial_1} & {I_2} & {I_3}
	\arrow[dashed, from=1-1, to=1-2]
	\arrow[dashed, from=1-2, to=1-3]
	\arrow["\pi", dashed, from=1-3, to=1-4]
	\arrow[dashed, from=1-3, to=3-3]
	\arrow[dashed, from=1-4, to=1-5]
	\arrow["h"', dashed, from=1-4, to=3-4]
	\arrow["h", from=1-4, to=3-5]
	\arrow[from=3-1, to=3-2]
	\arrow[from=3-2, to=3-3]
	\arrow["{\partial_1}", from=3-3, to=3-4]
	\arrow[hook, from=3-4, to=3-5]
	\arrow[from=3-5, to=3-6]
\end{tikzcd}
\]
where only the lower exact row is necessarily $V$-definable. Now it is clear that the splitting of the upper sequence within $V(\R)$ (hence the splitting of $\pi$) gives a map $\R \to I_1 \times_{\im \partial_1} \R \to I_1$ which lifts $h$. Thus it suffices to show that every exact sequence of the form of the top row splits. We now turn to showing this.

\subsection{A New Choice Principle for the Solovay Model}

We prove in this section the following choice principle for $V(\R)$, which to our knowledge does not occur in the literature.

\begin{prop}
\label{choicefn}
In $V(\R)$, any $\R$-indexed family of countable sets has a choice function.
\end{prop}
\begin{proof}
Let $A=(A_r)_{r \in \R}$ be countable sets such that for some formula $\varphi$, parameter $r_0\in\mathbb{R}$, $a\in V$, and ordinal $\gamma$ such that $A_r=\{y\mid V_\gamma\models\varphi(y,r,r_0,a)\}$. 
Let $\lambda>\max(\kappa,\gamma)$ be regular.
We show that for any $r\in\R$, there is a formula $\psi$ and parameter $b\in V_\lambda$ such that, viewed as a class function in $V_\lambda^{V[G]}$, $\psi(r,b)\in A_r$.
We will then have a choice function by mapping each $r\in\R$ to the element of $A_r$ defined by the lexicographically least possible pair $(b,\psi)$ with respect to some well-order of $V_\lambda$ (in $V$). 

Suppose not, and fix such an $r$. Crucially, $A_r \in V(\R)$, so there is a small forcing $\P = \overline{Coll(\omega, \alpha)}$ in $V[r, r_0]$ whose generic (named by $\dot g$) is coded by a real, and a definable function $\psi$ so that
\[
V_\lambda[r, r_0] \models \,\forces_\P \forces_{\K}\varphi(\psi(\check{\dot g}),r,r_0,a)
\]
Now we claim
\begin{subclaim}
If $g_0$, $g_1$ are mutually generic over $\P$, $\psi(\dot g_0) \neq \psi(\dot g_1)$.
\end{subclaim}
\begin{proof}
Suppose not. 
By homogeneity of $\K$, there is a $p\otimes q\in\P\otimes\P$ such that 
\[p\otimes q\Vdash^{V_\lambda[r,r_0]}_{\P\otimes\P}\Vdash_{\K}\psi(\check{{\dot {g_0}}})=\psi(\check{{\dot {g_1}}}). \]
Then $``\psi(g)$ for any $V[r,r_0]$ generic $g\ni p$'' defines a well-defined member of $A_r$ since for any generics $g,g'\ni p$, there is a generic $h\ni q$ such that $(g,h)$ and $(g',h')$ are both mutually generic pairs.
Since $p\in V_\lambda$, this contradicts our hypothesis that there was no element of $A_r$ definable using $r,r_0$, and parameters in $V_\lambda$.
\end{proof}
\subclaimsdone
However, in $V(\R)$, if $\vec g$ is a countable sequence of generics for $\P$, there is a $\beta<\kappa$ such that $r,r_0,\vec g\in V[G\upharpoonright\beta]$. 
Then the tail forcing adds a generic for $\P$ over $V[G\upharpoonright\beta]$. 
In particular, for any countable sequence of generic filters for $\P$, there is a new generic that is pairwise mutually generic with each of them. Thus the set
\[
\{\psi(g) \mid g \in V(\R), g\text{ is } V[r, r_0]\text{-generic for }\P \}
\]
must be uncountable, yet all of its elements belong to $A_r$. This contradicts the countability of $A_r$, and completes the proof.
\end{proof}

Note that crucial in the above proof is the fact that $A_r$ is contained in $V(\R)$, and in particular contains \textit{some} elements which are definable from reals. In general, whether a countable $V$-definable set in $V^\K$ must have a $V$-definable element is open. The ordinal-definability version of this question was investigated by Kanovei--Lyubetsky \cite[Theorem 1.2(ii)]{kanovei}.

\subsection{The Splitting}

In this section we prove theorem \ref{extthm}. The crucial points have already been established, namely that a cocycle can be defined due to the above choice function, and that any cocycle is measurable, due to living in $V(\R)$. Thus let
\[
0 \to \Z \hookrightarrow A \xrightarrow{\pi} \R \to 0
\]
be in $V(\R)$, with $A, \pi$ definable from $r_0 \in \R$ and a parameter in $V$. From Proposition \ref{choicefn} we have that a choice function for $\pi$ is definable from $r_0$ and a parameter in $V$, namely there is some function
\[
s \colon \R \to A
\]
in $V(\R)$ that sections $\pi$. Now define the cocycle
\[
\varphi(a, b) = s(a+b) - s(a)-s(b)
\]
which is necessarily $\Z$-valued since it is in the kernel of $\pi$. Since it is a map $\R \times \R \to \Z$ which lives in $V(\R)$, it is (Lebesgue/Haar) measurable. Now work temporarily in the ambient model of choice, $V^\K$.

Let us make a few comments as to the status of $A$ under these hypotheses. Mackey \cite{Mackey} shows that since $\R$, $\Z$ are locally compact, one can induce a Borel structure on $A$ and an associated translation invariant Haar measure, so that the extension is split iff the above cocycle $\varphi$ is a.e. a coboundary (note the near magical elision of a measure zero set).

Furthermore, every measurable cocycle corresponds, by work of Moore \cite{Moore}, to a strict exact sequence of locally compact Polish abelian groups:
\[
0 \to \Z \hookrightarrow A' \xrightarrow{\pi'} \R \to 0
\]
so that the cocycle defined from $\pi'$ is also $\varphi$ a.e.. Strict here means that the inclusion of $\Z \hookrightarrow A'$ is a closed embedding and $\pi'$ is an open surjection. However, in this setting, $\R$ is projective (see for instance Moskowitz \cite[Theorem 3.3]{Moskowitz}), hence $\pi'$ has a continuous splitting, hence $\varphi$ is equal to a coboundary a.e., hence $\pi$ has a splitting as well. The coboundary can be chosen to be Borel over $V[r_0]$, hence it and the splitting live also in $V(\R)$. Thus we have a splitting in $V(\R)$.
\pagebreak
\section{Concluding Remarks}
\label{questions}
The connection between condensed mathematics and set theory elaborated in this paper raises many further questions. For instance, it is natural to wonder what is preserved by the passage from $\pyk$ to $\VD$. 
\begin{quest}
    Which formulae in the internal language of $\pyk$ are preserved by the functor $\Lambda$?
\end{quest}

Furthermore, there are a number of useful objects built out of condensed sets, including analytic rings and liquid vector spaces.
One might be tempted to ask whether such objects have analogues in the Solovay model and what the functor~$\Lambda$ does with these objects. 
\begin{quest}
    What kind of object is the image of an analytic ring under $\Lambda$? 
    In particular, what happens to the $p$-adic analytic ring?
\end{quest}

As with the case of discrete Whitehead groups, we may ask if applying the functor $\Lambda$ can tell us more about $\pyk$ than we have shown. 
Many more pleasing measurability facts are true in the Solovay model than we have applied in this paper. The computation of the extension group in $\VD$ also suggests that group cohomology with measurable cocycles may be computed ``by passing to the Solovay model".
\begin{quest}
    Can we apply automatic continuity and other measurability phenomena in the Solovay model to answer more questions about $\pyk$ or $\cond$? In particular, can they be used to compute other derived functors?
\end{quest}

In fact, as with any Grothendieck topos, $Sh(\sol)$ has its own notion of sheaf cohomology, though it is unclear if these groups have any topological significance. Therefore, one might ask:

\begin{quest}
Does $i^*$ induce an isomorphism between the sheaf cohomologies of $\underline{X}$ and $i^*\underline{X}$ with coefficients in $\underline\Z$ and $i^*\underline\Z$ respectively?
\end{quest}

Note that the above mentions $i^*$, and not $\Lambda$, as we believe it is unreasonable to expect such an isomorphism with the more ``fragmented" cohomology theory of $\VD$. There is a sense in which passing to the double negation subtopos is what costs us the most fidelity between $\pyk$ and the Solovay model, since the functor~$i^*$ preserves a great deal more structure. This leads to the question of whether set theoretic tools can be used to directly study $Sh(\sol)$. This topos is of course intuitionistic, and hence does not correspond to a ``classical" symmetric extension.

\begin{quest}
What is the right construction of the Solovay model in the intuitionistic context? Can it be shown that all sets of reals have the Baire property in this model?
\end{quest}

In the converse direction, as noted in Example \ref{exub}, $\pyk$ seems to more naturally host constructions such as the universally Baire sets (which are, after all, invariant under the action of \textit{all} continuous maps). This might also lead one to wonder if $\pyk$ holds information of specific relevance for set theory, and that the information of the action of all continuous maps provides even more nuance than is available in the Solovay model (or even in $Sh(\sol)$).

\begin{quest}
What can $\pyk$ tell us about universally Baire sets of reals? Do determinacy hypotheses affect the theory of $\pyk$?
\end{quest}

A question that we have left untouched, but is of deep relevance to set theory, is the large cardinal strength required of $\kappa$ for pyknotic constructions to mirror condensed ones. For instance, does the way that $\pyk$ ``approaches" $\cond$ as $\kappa$ increases correspond to a finite support iteration of all forcings?
\begin{quest}
Let $\sol_\infty$ and $\sol_{\infty, \neg\neg}$ be obtained from $\sol$ and $\sol_{\neg\neg}$ by changing $({<} \kappa)$ to ``set-sized", and $\VD_{\infty}$ be the category of $V$-definable maps in the forcing extension by the class collapse forcing $Coll(\omega, {<} Ord)$. Is there a similar sequence of exact cocontinuous functors %(geometric morphisms in the sense of Shulman \cite[Remark 9.9]{ECSS}, since these are not topoi)
\[
\cond \to Sh_{small}(\sol_\infty) \to Sh_{small}(\sol_{\infty, \neg\neg}) \to \VD_{\infty}?
\]
\end{quest}
See Krapf \cite{CFSOA} for an extensive exploration of the class forcing extension by $Coll(\omega, {<} Ord)$. One may also wonder about approaching $\cond$ using small sites of large cardinal size, a question which is not \textit{a priori} about $\sol$ at all:

\begin{quest}
If $\kappa$ is supercompact, does the elementary subtopos of $\pyk_\kappa$ given by sheaves of size ${<} \kappa$ permit a cocontinuous Heyting functor to/from $\cond$?
\end{quest}

In another direction, Clausen--Scholze more recently turn to the topos of ``light" condensed sets \cite{analyticpdf}, constructed on the site of Stone spaces of countable weight. Some of the proofs in this paper seem to adapt to the light context, though one must be careful to use the site of all countably \textit{generated} complete Boolean algebras.

%, as in Remark \ref{r_star_remark}, to use $V(\R^*)$ (the hull over the set of reals added in some intermediate extension).

\begin{quest}
Which symmetric extension corresponds to the topos of light condensed sets via the methods of this paper?
\end{quest}

On the other hand, there are many ways to produce a Boolean topos ``from" $\pyk$, and our choice in this regard may seem \textit{ad hoc}. We initially also considered the site with maps dual to countably complete homomorphisms between countably complete Boolean algebras with the finite covering topology. This class of spaces is known in the literature as the basically disconnected spaces (see Vermeer \cite{vermeer}). The topos of sheaves on this site has an even closer relationship to $\pyk$, as it requires a larger family of continuous maps to act on its data, but we were unable to find a suitable model of set theory to which its double negation topos may be mapped. In addition, the theory of absolutes (i.e. projectives) and pullbacks between maps in this site is not as well-behaved.

%In fact, one might define an entire family of topoi interpolating between $\pyk$ and $\sol$ by considering the continuous maps dual to $\lambda$-complete homomorphisms with $\mu$-sized jointly universally effective epimorphic families for every pair of cardinals $\lambda,\mu<\kappa$. These are, in turn, associated to an infinite family of models of set theory which are adapted in different ways to continuous information.

As our understanding of basically disconnected spaces grows, it may become feasible to understand the sheaves on this site, its associated model of set theory, and its relationship to $\pyk$. Alternatively, it may be the case that $Sh(\sol)$ is already able to reflect more of the behaviour of $\pyk$ than we have thus far surmised. To us, this indicates that the search for the \textit{right} model of set theory in which to situate the proofs of condensed mathematics is far from complete.

\section*{Acknowledgements}
The second author's research was supported by the European Union’s Horizon 2020 research and innovation programme under the Marie Sk\l odowska-Curie grant agreement number 945322.

The authors would like to thank Reid Barton, Jeffrey Bergfalk, Ben de Bondt, James Cummings, Asaf Karagila, Chris Lambie-Hanson, Morgan Rogers, Boban Velickovic, and Alessandro Vignati, for discussions and various insights into topos theory, set theory, and the subject of this paper.
\pagebreak
\section*{Appendix}

\setcounter{section}{0}
\renewcommand{\thesection}{\Alph{section}}
\section{Dealing in Absolutes}

We prove here certain results that we promised in Proposition \ref{maps_facts}, as well as recount relevant details from the theory of extremally disconnected spaces. We start by noting for good measure that being skeletal (the preimage of an open dense set is open dense) is always weaker than being \qo{}, but in our context the two will coincide.

\begin{lemma}
\label{skeletal-is-qo}
If $f \colon X \to Y$ is skeletal with $X$ compact Hausdorff and $Y$ Hausdorff, then $f$ is \qo{}.
\end{lemma}
\begin{proof}
The crucial points are that $f$ is closed and that $X$ is regular. Let $U \subseteq X$ be a non-empty open. Since $X$ is regular, there is a nonempty open set $W$ such that $W\subseteq\overline{W}\subseteq U$. Now note
\[
f(W) \subseteq f(\overline{W}) \subseteq f(U)
\]
where the middle term is closed, hence $f(\overline{W}) \supseteq \overline{f(W)}$. Note that we have
\[
W \cap \inv{f}(Y \setminus \overline{f(W)}) = \varnothing
\]
hence $\inv{f}(Y \setminus \overline{f(W)})$ is not dense, hence $Y \setminus \overline{f(W)}$ cannot be open dense (since $f$ is skeletal). Therefore $\overline{f(W)}$ has non-empty interior. But then so does $f(U)$, hence $f$ is \qo{}.
\end{proof}

The next lemma implies that if $W$ is a dense subspace of an extremally disconnected space $S$, then the Stone--\v{C}ech compactification of $W$ is given by the inclusion of $W$ into $S$.

\begin{lemma}
\label{extensionlem}
If $W \subseteq S$ is dense, $S$ extremally disconnected compact Hausdorff, $X$ compact Hausdorff, and $h \colon W \to X$ is continuous, then there is a unique extension of $h$ to all of $S$. If $h$ is skeletal, then the extension is also skeletal (hence \qo{}).
\end{lemma}
\begin{proof}
Uniqueness is clear, since two continuous maps into a Hausdorff space must agree on a closed set. For existence, recall the following criterion for $h$ to extend to $S$ (due originally to Ta\u{i}manov \cite{taimanov}; see also Engelking \cite[Theorem 3.2.1]{engelking}): we must show that for $C_1$, $C_2$ disjoint closed in $X$, $h^{-1}C_1$, $h^{-1}C_2$ have disjoint closures in $S$.

To this end let $C_1$, $C_2$ be disjoint closed in $X$. Since $X$ is compact Hausdorff, hence normal, there are disjoint open neighbourhoods $U_i \supseteq C_i$. Then $h^{-1}U_i$ are disjoint and open in $W$. Let $O\subseteq S$ be open with $O\cap W=h^{-1}U_1$. 
Since $S$ is extremally disconnected, $\overline{O}$ is clopen. 
Moreover, since $W$ is dense, $\overline{O}=\overline{O\cap W}$. 
In particular, $\overline{h^{-1}C_1}\subseteq \overline{O}$ and $\overline{h^{-1}C_2}\subseteq S\setminus\overline{O}$ are disjoint. 

 Thus $h$ has a continuous extension $\overline{h}\colon S \to X$. Now suppose $h$ is skeletal, and let $U \subseteq X$ be open dense. Then $\overline{h}^{-1}U$ contains $h^{-1}U$, which is dense in $W$, hence also dense in $S$. Hence $\overline{h}^{-1}U$ is open dense. Thus the extension is also skeletal, hence by Lemma \ref{skeletal-is-qo}, is \qo{}.
\end{proof}

% \begin{lemma}
% \label{gen-stone-quasi-open}
% If $f : X \to Y$ is a generic map between compact Hausdorff spaces, and $X$ is totally disconnected (has a basis of clopen sets), then $f$ is quasi-open, i.e. the image of every open set has non-empty interior.
% \end{lemma}
% \begin{proof}
% We note that every non-empty open set in $X$ contains a non-empty clopen subset, thus it suffices to check that the image of every non-empty clopen set has non-empty interior under $f$.

% Suppose for contradiction that $U \subseteq X$ is clopen, and $f(U)$ has empty interior. Since $U$ is clopen, it is also compact, hence so is $f(U)$, hence $f(U)$ is closed with empty interior. Thus
% \[
% \inv{f}(Y \setminus f(U))
% \]
% is open dense, but then it must intersect $U$, which is absurd.
% \end{proof}

\begin{prop}
\label{ext-disc-gen-open}
If $f\colon X \to Y$ is a \qo{} map between extremally disconnected compact Hausdorff spaces, then $f$ is open. 
\end{prop}
\begin{proof}
Let $U \subseteq X$ be clopen (it is enough to check on a basis element). Consider
\[
W = U \setminus \inv{f}(int \ f[U])
\]
which is clopen, since $Y$ is extremally disconnected. Its image,
\[
f[U] \setminus \ int\ f[U]
\]
is closed nowhere dense, hence $W$ must be empty. That is, $f[U] \subseteq int\ f[U]$, so $f[U]$ is open.
\end{proof}

We now turn to the facts concerning absolutes and \qo{} maps. We refer the interested reader to Ponomarev--Shapiro \cite{ponomarev-shapiro} for more on this topic. Gleason shows that the spaces $\A(X)$ are projective in the category of compact Hausdorff spaces, hence every map possesses at least one absolute. Next we check that $\pi_X \colon \A(X) \to X$ is \qo{} (surjectivity and irreducibility being standard, due to Gleason \cite{gleason}; see also the Stacks Project \cite[\href{https://stacks.math.columbia.edu/tag/08YH}{Tag 08YH}]{stacks}).

\begin{prop}
\label{pigen}
$\pi_X\colon \A(X) \to X$ is \qo{}.
\end{prop}
\begin{proof}
We check on the clopen basis. Each clopen of $\A(X)$ is of the form
\[
N_U = \{F \mid F \text{ an ultrafilter on }RO(X), U \in F\}
\]
for some ground model regular open $U$. The image of this set under $\pi_X$ is
\[
\pi_X[N_U] = \left\{\bigcap\left\{\overline{W} \mid W \in F\right\} \mid U \in F\right\}.
\]
It is clear that this is contained in $\overline{U}$. Conversely, if $x \in \overline{U}$, the filter $F_x$ of regular opens that contain $x$ is compatible with $U$, i.e.
\[
x \in W \r W \cap \overline{U} \neq \varnothing \r W \cap U \neq \varnothing
\]
since $U$ is regular.

Hence by the ultrafilter lemma there is an ultrafilter $F$ extending $F_x \cup \{U\}$ in $RO(X)$. If $W$ is a regular open and $x \notin \overline{W}$, then $x$ has a regular open neighbourhood contained in $\overline{W}^c$ (by the regularity of $X$), hence every element of $F$ contains $x$ in its closure. Thus $\pi_X(F) = x$, and in particular
\[
x \in \pi_X[N_U].
\]
Hence $\pi_X[N_U] = \overline{U}$ has non-empty interior (since $U$ is regular open).
\end{proof}

% \begin{proof}
% Let $W \subseteq X$ be open dense. Then $\pi_X^{-1}W$ is given by
% \[
% \{F \mid F\text{ an ultrafilter on }RO(X), \bigcap\{\overline{U} \mid U \in F\} \in W\}
% \]
% By the compactness of $X \setminus W$, if $\bigcap\{\overline{U} \mid U \in F\} \in W$, there must be some $U \in F$ with $\overline{U}$ contained entirely in $W$. Thus the above set is exactly
% \[
% \{F \mid \exists U \in F (\overline{U} \subseteq W)\}
% \]
% which is the union of the basic opens
% \[
% \bigcup_{\overline{U} \subseteq W, U \in RO(X)} N_U
% \]
% Showing this is dense amounts to showing that for any regular open $U$, there is $U' \subseteq U$ with $\overline{U'} \subseteq W$. But this follows by the fact that $U \cap W$ is non-empty and $X$ is regular (any point $U \cap W$ has a closed neighbourhood still contained therein).
% \end{proof}

\begin{lemma}
\label{uniqueabslem}
A \qo{} map $f \colon X \to Y$ for $X$, $Y$ compact Hausdorff has a unique absolute.
\end{lemma}
\begin{proof}
Due to Shapiro \cite{shapiro}, this follows from the condition $HJ$: the inverse image of the boundary of a regular open must be nowhere dense. It amounts to showing that for $U \subseteq Y$ regular open
\[
\inv{f}(\overline{U} \setminus U)
\]
has empty interior, but this follows from the fact that
\[
\inv{f}(U \cup \overline{U}^c)
\]
which is the inverse image of an open dense set, is dense (since $f$ is \qo{}).
\end{proof}

In what follows, for clarity, we explicitly denote $\overline{U}$ by $cl \ U$. We make the following \textit{ansatz} about the form of the dual to the unique absolute of the \qo{} map $f$:
\[
\varphi_f \colon U \mapsto int \ cl \ \inv{f}U,
\]
which clearly defines at least a function $RO(Y) \to RO(X)$. First we show that this defines a complete homomorphism of Boolean algebras, hence that it is dual to an open map. Recall that, in this notation, the operations of the algebra of regular opens become
\[
\neg U = int(U^c)
\]
and for disjunction over a family indexed by $\alpha$,
\[
\bigvee_\alpha U_\alpha = int \ cl \ \bigcup_\alpha U_\alpha.
\]
The lemma below restates the preservation of negation by our \textit{ansatz} $\varphi_f$.

\begin{lemma}
Let $X$, $Y$ be compact Hausdorff, $f \colon X \to Y$ \qo{}, and $U \subseteq Y$ a regular open set. Then
\[
int \ cl \ \inv{f} \ int (U^c) = int \ (int \ cl \ \inv{f} U)^c.
\]
\end{lemma}
\begin{proof}
Equivalently, we show
\[
int \ cl \ \inv{f} \ int (U^c) = int \ cl \ int \ \inv{f} (U^c).
\]
Since $\inv{f}(U^c)$ is closed, the right hand side above can be simplified as
\[
int \ cl \ int \ \inv{f} (U^c) = int \ \inv{f} (U^c)
\]
therefore we reduce to showing
\[
int \ cl \ \inv{f} \ int (U^c) = int \ \inv{f} (U^c).
\]
Since $int (U^c) \subseteq (U^c)$, we have $\inv{f} \ int (U^c) \subseteq \inv{f} (U^c)$, whence $cl \ \inv{f} \ int (U^c) \subseteq \inv{f} (U^c)$, and finally $int \ cl \ \inv{f} \ int (U^c) \subseteq int \ \inv{f} (U^c)$. In the other direction, we have to show that
\[
int \ \inv{f} (U^c) \subseteq cl \ \inv{f} \ int (U^c)
\]
or equivalently the opposite inclusion between their complements
\[
cl \ \inv{f} U \supseteq int \ \inv{f} \ cl \ U.
\]
Note that $int \ cl \ U = U$, hence $cl \ U = R \sqcup U$ where $R$ has empty interior. Since $U \cup (cl\ U)^c$ is open dense, $\inv{f}R$ has empty interior also by $f$ being \qo{}. Now suppose $W$ is an open set such that
\[
W \subseteq \inv{f} \ cl \ U = \inv{f}R \sqcup \inv{f} U.
\]
Since $\inv{f}R$ has empty interior, $W$ must intersect $\inv{f}U$. This shows every interior point of $\inv{f} \ cl \ U$ is a limit point of $\inv{f}U$. The conclusion follows.
\end{proof}

\begin{lemma}
\label{gen-absolutes-complete}
If $X$, $Y$ are compact Hausdorff spaces, and $f \colon X \to Y$ is \qo{}, then
\[
\varphi_f \colon U \mapsto int \ cl \ \inv{f}U
\]
is a complete homomorphism $RO(Y) \to RO(X)$.
\end{lemma}
\begin{proof}
Negation is covered above. We check preservation of arbitrary suprema. Let $\{U_\alpha\}_\alpha$ be an arbitrary family of regular opens of $Y$. First we show
\[
cl \ \bigcup_\alpha int \ cl \ \inv{f} U_\alpha = cl \ \inv{f} \bigcup_\alpha U_\alpha.
\]
This follows by noting that
\begin{itemize}
\item Let $C$ be a closed set containing $\bigcup_\alpha int \ cl \ \inv{f} U_\alpha$, then $C$ contains $int \ cl \ \inv{f} U_\alpha$, hence contains $\inv{f} U_\alpha$, hence $\bigcup_\alpha \inv{f} U_\alpha$.
\item Let $C$ be a closed set containing $\bigcup_\alpha \inv{f} U_\alpha$, then $C$ contains $\inv{f} U_\alpha$, hence $C \supseteq cl \ \inv{f}U_\alpha$, hence $C \supseteq int \ cl \ \inv{f} U_\alpha$, hence $C$ contains $\bigcup_\alpha int \ cl \ \inv{f} U_\alpha$.
\end{itemize}
Next we show that
\[
cl \ \inv{f} \bigcup_\alpha U_\alpha = cl \ \inv{f} \ int \ cl \ \bigcup_\alpha U_\alpha.
\]
This follows by noting
\begin{itemize}
\item $cl \ \inv{f}$ is monotone, and $\bigcup_\alpha U_\alpha \subseteq int \ cl \ \bigcup_\alpha U_\alpha$, so one direction is clear.
\item $int \ cl \ \bigcup_\alpha U_\alpha = R \sqcup \bigcup_\alpha U_\alpha$ where $R$ is a set with empty interior. Further, $R$ is disjoint from $\bigcup_\alpha U_\alpha \cup (cl \ \bigcup_\alpha U_\alpha)^c$, which is open dense. Hence $\inv{f}R$ has empty interior since $f$ is \qo{}, hence
\[
\inv{f} int \ cl \ \bigcup_\alpha U_\alpha = \inv{f}R \sqcup \inv{f} \bigcup_\alpha U_\alpha
\]
is the union of an open set and a set with empty interior. Any open set that intersects the open $\inv{f}int \ cl \ \bigcup_\alpha U_\alpha$ cannot have this intersection (which is open) contained in $\inv{f}R$, hence must also intersect $\inv{f}\bigcup_\alpha U_\alpha$. Hence the two have the same closure.
\end{itemize}
These two together show that
\[
int \ cl \ \bigcup_\alpha int \ cl \ \inv{f} U_\alpha = int \ cl \ \inv{f} \bigcup_\alpha U_\alpha = int \ cl \ \inv{f} \ int \ cl \ \bigcup_\alpha U_\alpha
\]
or in other words,
\[
\bigvee_\alpha \varphi_f(U_\alpha) = \varphi_f \left(\bigvee_\alpha U_\alpha\right).
\]
\end{proof}

\begin{prop}
\label{gen-absolutes-open}
Let $f \colon X \to Y$ between compact Hausdorff spaces be a \qo{} map. The Stone dual of
\[
\varphi_f \colon U \mapsto int \ cl \ \inv{f}U
\]
is the unique absolute of $f$.
\end{prop}
\begin{proof}
From Lemma \ref{gen-absolutes-complete}, we know the Stone dual to the above is an open continuous map. From Lemma \ref{uniqueabslem} we know that if the dual map of $\varphi_f$ is an absolute, it is unique. Thus we are left to check the commutativity of the square
\[\begin{tikzcd}
	{\A(X)} && {\A(Y)} \\
	\\
	X && Y
	\arrow["{S(\varphi_f)}", from=1-1, to=1-3]
	\arrow["{\pi_X}"', from=1-1, to=3-1]
	\arrow["{\pi_Y}", from=1-3, to=3-3]
	\arrow["f", from=3-1, to=3-3]
\end{tikzcd}\]
To this end, let $F$ be an ultrafilter of regular open sets of $X$ (i.e. a point of $\A(X)$). $S(\varphi_f)$ acts on $F$ by pushforward, i.e. it maps to
\[
\{U \in RO(Y) \mid int \ cl \ \inv{f}U \in F\}
\]
Thus we compute
\[
\pi_Y(S(\varphi_f)(F)) = \text{unique pt of }\bigcap\{cl \ U \mid int \ cl \ \inv{f}U \in F\}
\]
whereas $f(\pi_X(F)) = f(x)$ where $x$ is the unique point of
\[
\bigcap \{cl \ W \mid W \in RO(X), W \in F\}.
\]
Thus we need to see that for each $U \in RO(Y)$ so that $int \ cl \ \inv{f} U \in F$, $f(x) \in cl \ U$. Fix such a $U$. By the definition of $x$ and since $\inv{f}U$ is open, we know that
\[
x \in cl \ int \ cl \ \inv{f}U = cl\ \inv{f}U.
\]
Hence by the continuity of $f$,
\[
f(x) \in f[cl\ \inv{f}U] \subseteq cl f[\inv{f}U] \subseteq cl\ U.
\]
Hence the square commutes, and we are done.
\end{proof}

\section{The Who's Who of $Sh(\sol)$}

\begin{prop} \label{subcanonical_prop}
Both $\sol$ and $\sol_{\neg\neg}$ are subcanonical, i.e. all the representable presheaves on the sites are sheaves. More generally, presheaves on $\sol$ given by
\[
X^o(S) = \{f \colon S \to X \mid f \ \qo{}\}
\]
for $X$ compact Hausdorff, $S$ extremally disconnected Stone, are sheaves for both finite and dense coverages.
\end{prop}
\begin{proof}
Suppose we are given a cover $\{U_i\}_i \xrightarrow{f_i} S$ and a collection of \qo{} maps $g_i \colon U_i \to X$ which are a matching family, i.e. for any two open maps $h \colon V \to U_i$, $k \colon V \to U_j$, if $f_ih = f_j k$, then $g_i h = g_j k$.

The crucial point is that for every pair of indices $i, j$, the two maps
\[
\pi_1 \colon \A(U_i \times_S U_j) \to U_i \times_S U_j \to U_i
\]
and
\[
\pi_2 \colon \A(U_i \times_S U_j) \to U_i \times_S U_j \to U_j
\]
are open (where the pullback $U_i \times_S U_j$ is via $f_i, f_j$ in the continuous category), since they are composites of \qo{} maps and the domain and codomain are extremally disconnected. We note that
\[
f_i \pi_1 = f_j \pi_2
\]
by definition, hence
\[
g_i \pi_1 = g_j \pi_2.
\]

In particular, for $i = j$, $g_i \pi_1 = g_i \pi_2$ shows that $g_i$ induces a continuous map $h_i \colon f_i[U_i] \to X$ so that $h_i \circ f_i = g_i$ (explicitly the function $h_i$ induced by $g_i$ is continuous, since $f_i[U_i] \subseteq S$ has the quotient topology, since all spaces involved are compact Hausdorff). For $i \neq j$, $g_i \pi_1 = g_j \pi_2$ shows that these maps agree on the intersections, hence we get a well-defined unique continuous map $h \colon \bigcup_i f_i[U_i] \to X$.
We show this map is \qo{}.

If $W \subseteq \bigcup_i f_i[U_i]$ is non-empty open, then
\[
h[W] = \bigcup_i h_i[W \cap f_i[U_i]] = \bigcup_i h_i[f_i[f_i^{-1}W]] = \bigcup_i g_i[f_i^{-1}W]
\]
which has non-empty interior since at least one of the $f_i^{-1}W$ is non-empty open, as the images of the $f_i$ jointly cover $\bigcup_i f_i[U_i]$. In the finite case we are already done. In the infinite dense case we are done by Lemma \ref{extensionlem}, noting that a \qo{} map is in particular skeletal.

\end{proof}

We should note that the extension from ``open" to ``\qo{}" is not an extreme jump, especially from the point of view of this topos:
\begin{prop}
\[
\A(X)^o \simeq X^o
\]
via the map $(f \colon S \to \A(X)) \mapsto (\pi_X \circ f \colon S \to X)$.
\end{prop}
\begin{proof}
This just an expression of the existence and uniqueness of absolutes for \qo{} maps.
\end{proof}

We note the following, largely for insight into the significance of the properties of $i^*$ demonstrated in section \ref{bridge_section}.

\begin{lemma}
\label{product_map_epi}
The natural map
\[
Lan_i(F \times G) \to Lan_i F \times Lan_i G
\]
is an objectwise surjection, hence in particular an epimorphism.
\end{lemma}
\begin{proof}
The crucial point is, as before, that if $Y_1$, $Y_2$ are extremally disconnected, the maps $\pi_i : \A(Y_1 \times Y_2) \to Y_i$ are open. Next note that since products are computed pointwise
\[
Lan_i(F) \times Lan_i(G) = \{([u_1 \in F(Y_1), f_1 : X \to Y_1], [u_2 \in G(Y_2), f_2 : X \to Y_2])\}
\]
where $[u, hf] = [F(h)(u), f]$ if $X \xrightarrow{f} Y \xrightarrow{h} Y'$, $u \in F(Y')$, and $h$ open. Now we ``synchronise" these pairs. Let
\[
([u_1 \in F(Y_1), f_1 : X \to Y_1], [u_2 \in G(Y_2), f_2 : X \to Y_2]) \in Lan_i(F) \times Lan_i(G).
\]
Note that there is a unique $(f_1, f_2) : X \to Y_1 \times Y_2$, and since $X$ is extremally disconnected, this lifts to a (not necessarily open) map $h : X \to \A(Y_1 \times Y_2)$. By construction, noting that the $\pi_i$ are open:
\[
[u_i \in F(Y_i), f_i : X \to Y_i] = [u_i \in F(Y_i), \pi_i h]
\]
\[
= [F(\pi_i)(u_i) \in F(\A(Y_1 \times Y_2)), h : X \to \A(Y_1 \times Y_2)]
\]
However, this pair is now in the image of $Lan_i(F \times G) \to Lan_i(F) \times Lan_i(G)$ since
\[
[(F(\pi_1)(u_1), F(\pi_2)(u_2)), h]
\]
is a preimage. This shows that $Lan_i(F \times G) \to Lan_i(F) \times Lan_i(G)$ is an objectwise surjection.
\end{proof}

\begin{prop}
The natural map $i^*(A^B) \to i^*A^{i^*B}$ between internal exponentials is a monomorphism.
%When $A \simeq Ran_i A'$ or $B \simeq Lan_i B'$, the map is a split monomorphism.
\end{prop}
\begin{proof}
We check objectwise:
\[
Sh(\sol)(C, i^*A^{i^*B}) \simeq Sh(\sol)(C \times i^*B, i^*A) \simeq \pyk(Lan_i(C \times i^*B), A)
\]
whereas
\[
Sh(\sol)(C, i^*(A^B)) \simeq \pyk(Lan_i C, A^B) \simeq \pyk(Lan_iC \times B, A),
\]
with the map between them given by precomposition with the map
\[
Lan_i(C \times i^*B) \xrightarrow{(Lan_i(\pi_1), Lan_i(\pi_2))} Lan_i C \times Lan_i i^*B \xrightarrow{(id, \epsilon_B)} Lan_i C \times B.
\]
Since $i^*$ is faithful, $(id, \epsilon_B)$ is epi, and by Lemma \ref{product_map_epi}, $(Lan_i(\pi_1), Lan_i(\pi_2))$ is an epi. Hence precomposition with this composite induces a mono.
\end{proof}

We also compute the subobject classifier of $\sol$ and $\sol_{\neg\neg}$. Recall (or see \cite[Section III.7]{maclane-moerdijk}) that the subobject classifier of $Sh(\sol)$ is given by 
\[
\Omega_{\sol}(S(\B)) = \{\text{finite-closed sieves into }S(\B)\}
\]
where a sieve $C = \{f \colon X_f \to S(\B)\}$ is finite-closed iff for all $g \colon X \to S(\B)$ and all covers $\{f_i \colon X_i \to X \colon i \in I\}$, if
\[
\{g \circ f_i \mid i \in I\} \subseteq C
\]
then $g \in C$. First note that $f \colon X_f \to S(\B)$ lies in $C$ iff $Im(f) \hookrightarrow S(\B)$ lies in $C$. Furthermore, if $\{f_i \colon X_i \to X_f\}$ is a cover of $X_f$, then $\{f \circ f_i \colon X_i \to Im(f)\}$ is a cover of $Im(f)$, hence also $\{Im(f \circ f_i) \hookrightarrow S(\B) \mid i \in I\}$ covers $Im(f)$. Thus the family of subobjects in any closed sieve uniquely determines the sieve.

\begin{prop}
$\Omega_\sol(S(\B))$ is in bijection with the set of ideals of $\B$.
\end{prop}
\begin{proof}
The bijections in question are $\varphi$ given by
\[
C \mapsto \{p \in \B \mid \text{the inclusion } S(\B_p) \hookrightarrow S(\B) \in C\}
\]
and $\psi$ given by
\[
I \mapsto \text{ sieve generated by }\{S(\B_p) \hookrightarrow S(\B) \mid p \in I\}
\]
That $\varphi(C)$ is an ideal and that $\psi(I)$ is a finite-closed sieve follows from the discussion above. Suppose $S(\B_q) \hookrightarrow S(\B) \in \psi(I)$, then we have
\[
\begin{tikzcd}
	{S(\B_q)} && {S(\B)} \\
	& {S(\B_p)}
	\arrow[hook, from=1-1, to=1-3]
	\arrow["f", from=1-1, to=2-2]
	\arrow[hook, from=2-2, to=1-3]
\end{tikzcd}
\]
for some $p \in I$. Since the other two maps are inclusions, so too must $f$ be, whence $q \leq p$, and $q \in I$. Thus $\varphi(\psi(I)) = I$.

Now suppose $f \in C$. As above, $Im(f) \simeq S(\B_p)$ and $S(\B_p) \hookrightarrow S(\B) \in C$. Hence $f$ is in the sieve generated by $S(\B_p) \hookrightarrow S(\B)$, hence $\psi(\varphi(C)) = C$.
\end{proof}

We may reason similarly with the double negation subtopoi of these topoi (where we now have infinite dense covers).

\begin{prop}
\label{doubleneg}
\[
\Omega_{\sol_{\neg\neg}}(S(\B)) \simeq \B (\simeq \text{complete ideals in }\B).
\]
\end{prop}
\begin{proof}
Mutatis mutandis, noting that one can simply take the closure of the union of elements of a ``dense-closed" sieve to obtain a single generator.
\end{proof}

From this the verification that $Sh(\sol_{\neg\neg})$ is the double negation topos of $Sh(\sol)$ is standard. Note first that a finite dense cover is a finite cover, hence $Sh(\sol_{\neg\neg})$ is a full subcategory of $Sh(\sol)$. The double negation subtopos of $Sh(\sol)$ is the unique full subtopos whose subobject classifier is the double negation sublocale of $\Omega_{\sol}$. The above proposition shows that $\Omega_{\sol_{\neg\neg}}$ is the double-negation sublocale of $\Omega_{\sol}$, which shows that $Sh(\sol_{\neg\neg})$ is this double-negation subtopos.

Since we know that $Sh(\sol_{\neg\neg}) \simeq \VD$, the above can be translated to give the subobject classifier of $\VD$. For completeness we record it here.
\begin{prop}
$\Omega_{\VD}$ is the (any) name for $\{0, 1\}$.
\end{prop}

\noindent We also note that internal exponentials in $\VD$ are those computed in $V(\R)$:
\begin{prop}
If $\dot a$, $\dot b$ are objects in $\VD$, then $\dot a^{\dot b}$ is the (any) name for the $V$-definable set
\[
\{f : a \to b \mid f \in V(\R)\}.
\]
\end{prop}
\begin{proof}
Any name for a $V$-definable map $\dot f : \dot c \times \dot a \to \dot b$ is equivalently a $V$-definable map $\dot g : \dot c \to \dot{a}^{\dot b}$ given by
\[
\forces_\K \forall x, y(\dot g(x)(y) = \dot f(x, y)).
\]
\end{proof}
Note that we really do mean all maps in $V(\R)$, not just $V$-definable ones, as for each $x$ the value of $g(x)$ is not necessarily $V$-definable, but only $V$-definable with the parameter $x$. From this it also follows that power objects in $\VD$ are simply the (names for) power sets taken in $V(\R)$, and hence the cumulative hierarchy of $\VD$ reproduces the class (of names for elements of) $V(\R)$.

\bibliography{refs}{}
\bibliographystyle{myalphastyle}

\iffalse

\fi

\noindent (N.~Bannister) \textsc{Department of Mathematical Sciences, Carnegie Mellon University, 5000 Forbes Ave,
    Pittsburgh, Pennsylvania 15213}\par\nopagebreak
  \textit{Email address}: \texttt{nathanib@andrew.cmu.edu}
\newline\newline
(D.~Basak) \textsc{Institut de Math\'ematiques Jussieu -- Paris Rive Gauche, Universit\'e~Paris~Cit\'e, 8 Pl. Aur\'elie Nemours, Paris 75013} \par\nopagebreak
\textit{Email address}: \texttt{basak@imj-prg.fr}
\end{document}